\documentclass{amsart}

\makeatletter
\let\@cleartopmattertags\relax
\newcommand\articleend{%
	\enddoc@text
	\let\authors\@empty
	\let\shortauthors\@empty
	\let\contribs\@empty
	\let\xcontribs\@empty
	\let\toccontribs\@empty
	\let\addresses\@empty
	\let\thankses\@empty
	\newpage
}
\let\@wraptoccontribs\wraptoccontribs
\makeatother

\usepackage{amsmath, amssymb, amsthm, mathtools, tikz-cd, xfrac}
\usepackage{booktabs}
\usepackage[hidelinks]{hyperref}

\theoremstyle{plain}
\newtheorem{theorem}{Theorem}
\newtheorem{proposition}[theorem]{Proposition}
\newtheorem{lemma}[theorem]{Lemma}
\newtheorem{corollary}[theorem]{Corollary}

\theoremstyle{definition}
\newtheorem{definition}[theorem]{Definition}
\newtheorem{remark}[theorem]{Remark}

\newtheorem*{acknowledgements}{Acknowledgements}

\newcommand{\Hom}{\mathrm{Hom}} 
\newcommand{\Homah}[2]{{\rm H}(#1,#2)} 
\newcommand{\indC}{\mathsf{ind}\text{-}\mathsf{C^*}}
\newcommand{\ind}[1]{\mathsf{ind}\text{-}\mathsf{#1}}
\newcommand{\cstar}{\mathsf{C^*}}
\newcommand{\ho}{\mathsf{ho}(\mathsf{C^*)}}
\newcommand{\hoind}{\mathsf{ho}(\mathsf{ind}\text{-}\mathsf{C^*})}
\newcommand{\ssh}{\mathsf{s}\text{-}\mathsf{sh}}
\newcommand{\sh}{\mbox{$\mathsf{sh}$} }
\newcommand{\hlim}{\mathrm{h}\text{-}\!\varinjlim}
\newcommand{\indho}{\mathsf{ind}\text{-}\mathsf{ho}(\mathsf{C^*})}
\newcommand{\hoforget}{\mathrm F}
\newcommand{\hoinclude}{\mathrm{Ho}}
\newcommand{\hausdorffize}{\mathrm{Hd}}
\newcommand{\asinclusion}{\mathrm{As}}
\newcommand{\AM}{\mathsf{AM}}
\newcommand{\AMH}{\mathsf{AM_{\mathrm{Hd}}}}
\newcommand{\ev}{\mathrm{ev}}
\newcommand{\id}{\mathrm{id}}

\DeclareRobustCommand{\SkipTocEntry}[5]{}

\begin{document}
	
\title{A topology on $E$-theory}
\author[J. Carri\'on]{Jos\'e R.\ Carri\'on}
\address{Jos\'e R.\ Carri\'on, Department of Mathematics, Texas Christian
	University, Fort Worth, Texas 76129, USA}
\email{j.carrion@tcu.edu}

\author[C. Schafhauser]{Christopher Schafhauser}
\address{Christopher Schafhauser, Department of  Mathematics, University of Nebraska - Lincoln, Lincoln, Nebraska, USA}
\email{cschafhauser2@unl.edu}

\date{April 15, 2024}


\begin{abstract}
  For separable $C^*$-algebras $A$ and $B$, we define a topology on the set $[[A, B]]$ consisting of homotopy classes of asymptotic morphisms from $A$ to $B$.  This gives an enrichment of the Connes--Higson asymptotic category over topological spaces.  We show that the Hausdorffization of this category is equivalent to the shape category of Dadarlat.  As an application, we obtain a topology on the $E$-theory group $E(A, B)$ with properties analogous to those of the topology on $KK(A, B)$.  The Hausdorffized $E$-theory group $EL(A, B) = E(A, B) / \overline{\{0\}}$ is also introduced and studied.  We obtain a continuity result for the functor $EL(\,\cdot\,,B)$, which implies a new continuity result for the functor $KL(\,\cdot\,,B)$.
\end{abstract}

\maketitle

\setcounter{tocdepth}{1}
\tableofcontents

\renewcommand{\thetheorem}{\Alph{theorem}}

\addtocontents{toc}{\SkipTocEntry}
\section*{Introduction}

$E$-theory was defined by Connes and Higson in \cite{Connes-Higson90}.
It is a concrete realization, defined in terms of asymptotic
morphisms, of the half-exact bifunctor first defined by Higson in
\cite{Higson90a}.  In parallel with $KK$-theory, this bifunctor from the category of separable
$C^*$-algebras (and $^*$-homomorphisms) to the category of abelian
groups is stable, homotopy invariant, and possesses a composition
product.  In fact, for separable $C^*$-algebras $A$ and $B$, there is a natural
transformation $E(A,B)\to KK(A,B)$ preserving the product structure
that is an isomorphism whenever $A$ is nuclear.  Unlike
$KK$-theory, however, $E$-theory is half-exact on all extensions of
separable $C^*$-algebras.  The theory played a prominent role in
approaches to both the Baum--Connes conjecture and classification
theory for nuclear $C^*$-algebras soon after its introduction
\cite{Higson-Kasparov01, Phillips00}.  See \cite{Connes-Higson90,
  Blackadar98, Guentner-Higson-etal00} for more on $E$-theory and its
applications.

In this paper, we define a topology on $E(A,B)$ with properties analogous to the ones satisfied by the topology on $KK(A,B)$.  With some restrictions on $A$ and $B$, the latter topology was first studied in depth by Brown \cite{Brown84} and Salinas \cite{Salinas92} in connection with quasidiagonal extensions---an early version is mentioned in the work of Brown, Douglas, and Fillmore \cite{Brown-Douglas-etal77}.  It was further developed by Schochet in \cite{Schochet01, Schochet02} before being defined and studied in general by Pimsner (unpublished) and Dadarlat \cite{Dadarlat05}.  That all of these definitions coincide (in their respective settings) relies on a characterization of convergence in terms of \emph{Pimsner's condition} (see \cite[Theorem~3.5]{Dadarlat05}).  The topology we define satisfies the $E$-theoretic version of this condition.

\begin{theorem}\label{thm:main}
  For separable $C^*$-algebras $A$ and $B$, there is a unique second countable topology on $E(A, B)$ such that $x_n \rightarrow x$ in $E(A, B)$ if and only if there exists $y \in E(A, C(\mathbb N^\dag, B))$ satisfying $y(n) = x_n$ for all $n \in \mathbb N$ and $y(\infty) = x$, where $\mathbb{N}^\dag = \mathbb N \cup \{\infty\}$ is the one-point compactification of the natural numbers. \end{theorem}

This and other properties of the topology are found in Section~\ref{sec:applications}.  In particular, if $A$, $B$, and $D$ are separable $C^*$-algebras, then $E(A, B)$ is a topological group and the product $E(A, B) \times E(B, D) \rightarrow E(A, D)$ is continuous, generalizing \mbox{\cite[Theorem~6.8]{Schochet01}} and in analogy with \cite[Theorem~3.5]{Dadarlat05}.

The relevance of the closure of $\{0\}$ was first highlighted in the work of Brown on the UCT in the early 1980s \cite{Brown84}.  When $A$ satisfies the UCT, the closure of $\{0\}$ can be identified with the subgroup of $\mathrm{Ext}_\mathbb{Z}^1(K_{*+1}(A), K_*(B))$ consisting of pure extensions, as proved by Schochet in \cite[Theorem~3.3]{Schochet02}, assuming nuclearity.  The nuclearity condition was removed by Dadarlat in \cite[Section 4]{Dadarlat05}.

When $A$ satisfies the UCT, the group $KL(A, B)$ is defined as the quotient of $KK(A,B)$ by this subgroup, and this quotient group plays a central role in classifying $^*$-homomorphisms, as first shown by R{\o}rdam \cite{Rordam95}.  In particular, R{\o}rdam proved that two $^*$-homomorphisms that are approximately unitarily equivalent induce the same $KL$-class, while they might not induce the same $KK$-class.  The $KL$-groups remain an indispensable tool in the classification program for nuclear $C^*$-algebras to this day (for example, see \cite{Gong-Lin-etal20, Gong-Lin-etal20a, TWW, cgstw23}). Moreover, the $KL$-groups capture the limiting behavior of the controlled $KK$-theory groups of Willett and Yu \cite[Theorem~1.2]{Willett-Yu21}.

Without the UCT condition, Dadarlat defined $KL(A,B)$ as the quotient of $KK(A,B)$ by the closure of $\{0\}$.
We define the Hausdorffized $E$-theory group $EL(A, B)$ in an analogous way.
This is a totally disconnected, separable, and completely metrizable topological group, just like $KL(A,B)$ (see Section~\ref{sec:e-theory}).
Further, the product on $E$-theory descends to a continuous product on $EL$-theory.
We also prove that two separable $C^*$-algebras are $E$-equivalent if and only if they are $EL$-equivalent.
This generalizes and gives a new proof of \cite[Corollary~5.2]{Dadarlat05}, stating that $KL$- and $KK$-equivalence coincide for nuclear $C^*$-algebras, with an argument that avoids the use of the Kirchberg--Phillips classification theorem.

We also examine the behavior of Hausdorffized $E$-theory
under direct limits.  Milnor's $\lim^1$-sequence provides
compatibility of $E$-theory with direct limits in the first variable:
for an inductive system $(A_n)_{n=1}^\infty$ of separable
$C^*$-algebras, there is a natural exact sequence
\[
  \begin{tikzcd}
    0 \arrow{r} & \varprojlim^1 E(A_n, SB) \arrow{r} & E(\varinjlim
    A_n, B) \arrow{r} & \varprojlim E(A_n, B) \arrow{r} & 0
  \end{tikzcd}
\]
where $SB = C_0(\mathbb R) \otimes B$ denotes the suspension of $B$.
(See \cite[Theorem 7.1]{Schochet84a} for a more general treatment.)
It turns out that the $\varprojlim^1$ term is always mapped into the
closure of $\{0\}$, which leads to the continuity of the functor
$EL(\,\cdot\,, B)$.  As a consequence, we obtain a new continuity
result for $KL$.

\begin{theorem}\label{thm:cont-EL}
  If $(A_n)_{n=1}^\infty$ is an inductive system of separable
  $C^*$-algebras and $B$ is a separable $C^*$-algebra, then the natural map
  \begin{equation*}
    EL\big(\varinjlim A_n, B\big) \longrightarrow \,\varprojlim EL(A_n, B)
  \end{equation*}
  is an isomorphism.  In particular, if each $A_n$ is nuclear, then
  the natural map
  \begin{equation*}
    KL\big(\varinjlim A_n, B\big) \longrightarrow \,\varprojlim KL(A_n, B)
  \end{equation*}
  is an isomorphism.
\end{theorem}

While we have stated our main results in terms of $E$-theory, the main
body of the paper is actually concerned with the set $[[A,B]]$ of
homotopy classes of asymptotic morphisms from $A$ to $B$ (see
Section~\ref{sec:prelim-asymp}).  As we remind the reader in
Section~\ref{sec:e-theory}, $E(A,B)$ is defined as $[[SA\otimes
\mathcal{K}, SB\otimes \mathcal{K}]]$, where $\mathcal K$ denotes the $C^*$-algebra of compact operators on a separable infinite dimensional Hilbert space.  However, the group structure
furnished by stabilizing and suspending is not relevant to our
development of the topology in Theorem~\ref{thm:main}.  The topology on $[[A, B]]$ is introduced in Section~\ref{sec:top}, and its Hausdorffization is studied in Section~\ref{sec:Hd}.

The various definitions of the topology on $KK(A,B)$ are typically
built using representation theoretic pictures of $KK$-theory.  In the
absence of such a description of $E$-theory, we opt for a new approach
that employs the shape theoretic methods of \cite{Dadarlat94} and a
generalization of Blackadar's homotopy lifting property for
semiprojective $C^*$-algebras from \cite{Blackadar16} (see
Theorem~\ref{thm:lifting}).  Shape theory for $C^*$-algebras was first
introduced by Effros and Kaminker in \cite{Effros-Kaminker86} and
refined by Blackadar in \cite{Blackadar85}.  Building on the
techniques of \cite{Dadarlat94}, which related shape theory and
$E$-theory, we prove the following result. (See
Section~\ref{sec:applications} for the relevant definitions.)

\begin{theorem}
  \label{thm:haus-shape}
  There is an equivalence between the shape category and the
  Hausdorffized asymptotic morphism category.
\end{theorem}

Finally, we give a brief heuristic description of the new methods
developed in this paper.  Dadarlat showed in \cite{Dadarlat94} that given a
homotopy commuting diagram
\[
  \begin{tikzcd}
    \cdots \ar[r] & A_{n-1} \ar[d] \ar[r] & A_n \ar[d] \ar[r] & A_{n+1} \ar[d] \ar[r] & \cdots &[-2.7em] \varinjlim A_n \ar[d, dashed, "\approx"]\\
    \cdots \ar[r] & B_{n-1} \ar[r] & B_n \ar[r] & B_{n+1} \ar[r] & \cdots & \varinjlim B_n   \end{tikzcd}
\]
together with a distinguished choice of homotopy at each stage, there is a limiting asymptotic morphism $\varinjlim A_n \xrightarrow\approx \varinjlim B_n$.  We prove
that this homotopy limit is independent of the choice of homotopies up to Hausdorffization (Proposition~\ref{prop:homotopy-limit}).  Once restricted to the setting where the
inductive systems are shape systems, this homotopy limit functor
provides the equivalence in Theorem~\ref{thm:haus-shape}.  In fact,
the systematic use of this idea is behind many of the results in this
paper.

\begin{acknowledgements}
	The second author was partially supported by NSF Grant DMS-2000129.  Each author would like to thank the other's university for their hospitality during the visits where much of this project was completed.
\end{acknowledgements}

\setcounter{theorem}{0}
\numberwithin{theorem}{section}

\section{Asymptotic morphisms and semiprojectivity}\label{sec:prelim}

This preliminary section serves to set our notation and recall some results from \cite{Dadarlat94} relating asymptotic morphisms to shape theory.  The only somewhat new result is Corollary~\ref{cor:homotopy-stability}, which is a slight variation of Blackadar's result in \cite[Corollary~4.3]{Blackadar16} and has nearly the same proof.

\subsection{Asymptotic morphisms}\label{sec:prelim-asymp}

For $C^*$-algebras $A$ and $B$, an \emph{asymptotic morphism} $\phi
\colon A \xrightarrow\approx B$ is a collection of
self-adjoint\footnote{That is, $\phi(a^*)=\phi(a)^*$ for all $a\in A$.} linear maps $(\phi_t \colon A \rightarrow B)_{t \geq 0}$, indexed by the space $\mathbb R_+$ of non-negative real numbers, such that
\begin{enumerate}
	\item $t \mapsto \phi_t(a)$ is continuous for all $a \in A$, and
	\item $\displaystyle \lim_{t \rightarrow \infty} \|\phi_t(ab) - \phi_t(a) \phi_t(b)\|$ for all $a, b \in A$.
\end{enumerate}

Note that every $^*$-homomorphism $\phi \colon A \rightarrow B$ may be
viewed as an asymptotic morphism that is constant in $t$.  If $A$, $B$, and $D$ are $C^*$-algebras,  $\phi \colon A \xrightarrow\approx B$ is an asymptotic morphism, and $\psi \colon B \rightarrow D$ is a $^*$-homomorphism, we may define an asymptotic morphism $\psi \circ \phi \colon A \xrightarrow\approx D$ by $(\psi \circ \phi)_t = \psi \circ \phi_t$.  Similarly, if $\psi \colon D \rightarrow A$ is a $^*$-homomorphism, we have an asymptotic morphism $\phi \circ \psi \colon D \xrightarrow\approx B$ given by $(\phi \circ \psi)_t = \phi_t \circ \psi$.

We emphasize that the $\phi_t$ are not assumed to be bounded.  However, as
noted in \cite[p.~102]{Connes-Higson90} (see also
\cite[Proposition~25.1.3]{Blackadar98}), an asymptotic morphism is always ``asymptotically contractive'' in the sense that
\begin{enumerate}\setcounter{enumi}{2}
	\item $\limsup\limits_{t \rightarrow \infty} \|\phi_t(a)\| \leq \|a\|$ for all $a \in A$.
\end{enumerate}
It follows that any asymptotic morphism $\phi \colon A \xrightarrow{\approx} B$ induces a $^*$-homomorphism $\phi_{\rm as} \colon A \rightarrow B_{\rm as}$, where $B_{\rm as} =C_b(\mathbb R_+, B) / C_0(\mathbb R_+, B)$ is the \emph{asymptotic algebra} of $B$.  Conversely, given a $^*$-homomorphism $A \rightarrow B_{\rm as}$ and a self-adjoint linear lift \mbox{$\Phi \colon A \rightarrow C_b(\mathbb R_+, B)$}, there is an induced asymptotic morphism $\phi \colon A \xrightarrow\approx B$ defined by \mbox{$\phi_t(a) = \Phi(a)(t)$} for $a \in A$ and $t \in \mathbb R_+$.  Then $\Phi$ lifts $\phi_{\rm as}$.

Two asymptotic morphisms $\phi, \psi \colon A \xrightarrow\approx B$
are \emph{equivalent}, written $\phi \cong \psi$, if
\begin{equation*}
	\lim_{t \rightarrow \infty} \|\phi_t(a) - \psi_t(a)\| = 0
\end{equation*}
for all $a \in A$.  Note that $\phi \cong \psi$ if and only if $\phi_{\rm as} = \psi_{\rm as}$.  We say $\phi$ and $\psi$ are \emph{asymptotically homotopic} if there is an asymptotic morphism $\theta \colon A \xrightarrow\approx C([0, 1], B)$ such that $\ev_0 \circ \theta \cong \phi$ and $\ev_1 \circ \theta \cong \psi$.  We let $[[\phi]]$ denote the equivalence class of an asymptotic morphism $\phi\colon A\xrightarrow{\approx} B$ under asymptotic homotopy and let $[[A, B]]$ denote the set of such equivalence classes.

There is no natural way of composing asymptotic morphisms, but there
is a well-defined composition up to asymptotic homotopy, as shown in
\cite[Proposition 4]{Connes-Higson90} (see also \cite[Theorem~25.3.1]{Blackadar98}).  The composition defined below extends the usual composition of $^*$-homomorphisms and, more generally, the composition of a $^*$-homomorphism with an asymptotic morphism discussed above.  This construction uses separability of the domain algebras and is the main reason separability hypotheses arise in our results.  (The statement in \cite[Theorem~25.3.1]{Blackadar98} asks that all three $C^*$-algebras be separable, but examining the proof shows that the separability of $D$ is not needed.)

\begin{theorem}\label{thm:asymptotic-category}
	If $A$, $B$, and $D$ are $C^*$-algebras with $A$ and $B$ separable and $\phi \colon A \xrightarrow \approx B$ and $\psi \colon B \xrightarrow\approx D$ are asymptotic morphisms, then there is a homeomorphism $r_0 \colon [0, \infty) \rightarrow [0, \infty)$ such that for all homeomorphisms $r \colon [0, \infty) \rightarrow [0, \infty)$ with $r(t) \geq r_0(t)$ for all $t \geq 0$, we have $(\psi_{r(t)} \circ \phi_t)_{t \geq 0}$ is an asymptotic morphism.  Moreover, the homotopy equivalence class of this asymptotic morphism, written $[[ \psi ]] \circ [[\phi]]$, is independent of the choice of $r$.
\end{theorem}

Following Connes and Higson \cite{Connes-Higson90}, we consider the category  $\AM$, henceforth called the \emph{asymptotic category}, whose objects are separable $C^*$-algebras, whose morphisms from $A$ to $B$ form the set $[[A, B]]$, and whose composition is given by Theorem~\ref{thm:asymptotic-category}.

\subsection{Semiprojectivity and shape systems}\label{sec:prelmin-shape}

We will write $(\underline{A}, \underline{\alpha})$ to denote an
inductive system of $C^*$-algebras
\begin{equation*}
	A_1 \xrightarrow{\alpha_1} A_2 \xrightarrow{\alpha_2} A_3 \xrightarrow{\alpha_3} \cdots.
\end{equation*}
As usual, for integers $n > m \geq 1$, we write 
\begin{equation*}
	\alpha_{n,m } = \alpha_{n-1} \circ \cdots \circ \alpha_m \colon A_m \rightarrow A_n.
\end{equation*}
For $n \geq 1$, we also write $\alpha_{n, n} = \mathrm{id}_{A_n}$ and $\alpha_{\infty, n}$ for the natural map $A_n \rightarrow \varinjlim\, (\underline{A}, \underline{\alpha}$).

Recall from \cite[Definition~2.10]{Blackadar85} that a
$^*$-homomorphism $\alpha \colon A_0 \rightarrow A$ between separable
$C^*$-algebras is \emph{semiprojective} if whenever
$(\underline{B}, \underline{\beta})$ is an inductive system with each
$\beta_n$ surjective and $\phi \colon A \rightarrow \varinjlim
(\underline{B}, \underline{\beta})$ is a
$^*$-homomorphism, there are an integer $n \in \mathbb N$ and a
$^*$-homomorphism $\tilde \phi \colon A_0 \rightarrow B_n$ such that
$\beta_{\infty, n} \circ \tilde \phi = \phi \circ \alpha$.  A
\emph{shape system} for a $C^*$-algebra $A$ is an inductive system
$(\underline{A}, \underline{\alpha})$ such that each $\alpha_n$ is
semiprojective and $\varinjlim\, (\underline{A}, \underline{\alpha})$ is
isomorphic to $A$.  Every separable $C^*$-algebra admits a shape
system $(\underline{A}, \underline{\alpha})$; in fact, one may take
each $\alpha_n$ to be surjective. See \cite[Theorem~4.3]{Blackadar85}.

If $A$, $B$, and $D$ are separable $C^*$-algebras, 
$\alpha \colon A \rightarrow B$ and $\beta \colon B \rightarrow D$ are
$^*$-homomorphisms, and either $\alpha$ or $\beta$ is
semiprojective, then $\beta \circ \alpha$ is also semiprojective.  We will
use this to show every semiprojective $^*$-homomorphism factors as a composition of two semiprojective $^*$-homomorphisms.

\begin{lemma}\label{lem:semiproj-factor}
	If $A_0$ and $A$ are separable $C^*$-algebras and $\alpha \colon A_0 \rightarrow A$ is a semiprojective $^*$-homomorphism, then there are a separable $C^*$-algebra $A_1$ and semiprojective $^*$-homomorphisms $\alpha_0 \colon A_0 \rightarrow A_1$ and $\alpha_1 \colon A_1 \rightarrow A$ such that $\alpha = \alpha_1 \circ \alpha_0$.
\end{lemma}

\begin{proof}
	Fix a shape system $(\underline A', \underline \alpha')$ for $A$ with each $\alpha'_n$ surjective.  Because $\alpha$ is semiprojective, there are an integer $n \geq 1$ and a $^*$-homomorphism $\tilde\alpha \colon A \rightarrow A'_n$ such that $\alpha = \alpha'_{\infty, n} \circ \tilde\alpha$.  Define $A_1 = A'_{n+1}$, $\alpha_0 = \alpha'_n \circ \tilde\alpha$, and $\alpha_1 = \alpha'_{\infty, n+1}$.  Then $\alpha_0$ and $\alpha_1$ are semiprojective since $\alpha'_n$ and $\alpha'_{n+1}$ are.
\end{proof}

Shape systems provide a convenient formalism for defining asymptotic
morphisms, as first noted in \cite{Dadarlat94}.

\begin{definition}
  \label{def:homotopy-morphisms}
  Given inductive systems $(\underline{A}, \underline{\alpha})$ and
  $(\underline{B}, \underline{\beta})$, a \emph{strong homotopy
    morphism}  $(f, \underline \phi,
  \underline h) \colon (\underline A, \underline \alpha)
  \xrightarrow\approx (\underline B, \underline \beta)$ consists of a
  strictly increasing map $f \colon \mathbb N \rightarrow \mathbb N$
  and sequences of $^*$-homomorphisms
  \[ \underline \phi = \big(\phi \colon A_n \rightarrow
    B_{f(n)}\big)_{n=1}^\infty \quad \text{and} \quad \underline h = \big(h_n \colon
    A \rightarrow C([0, 1], B_{f(n+1)})\big)_{n=1}^\infty \]
  such that for all $n \in \mathbb N$, $\ev_0 \circ h_n =
  \beta_{f(n+1), f(n)} \circ \phi_n$ and $\ev_1 \circ h_n =
  \phi_{n+1}\circ \alpha_{n+1, n}$. This can be visualized as a diagram
  \[
    \begin{tikzcd}[row sep = normal, column sep = 5pt]
      A_1 \arrow{rr}{\alpha_1} \arrow{dd}{\phi_1} \arrow{dr}{h_1} & & A_2 \arrow{rr}{\alpha_2} \arrow{dd}{\phi_2} \arrow{dr}{h_2} & & A_3 \arrow{rr}{\alpha_3} \arrow{dd}{\phi_3} \arrow{dr}{h_3} & & \cdots \\
      & C([0,1], B_{f(2)}) \arrow[shift left = 2pt]{dr}{\ev_1}
      \arrow[shift right =2pt ]{dr}[swap]{\ev_0} & & C([0, 1],
      B_{f(3)}) \arrow[shift left = 2pt]{dr}{\ev_1}
      \arrow[shift right =2pt ]{dr}[swap]{\ev_0}& & C([0, 1],
      B_{f(4)}) \arrow[shift left = 2pt]{dr}{\ev_1}
      \arrow[shift right =2pt ]{dr}[swap]{\ev_0}& \cdots \\
      B_{f(1)} \arrow{rr}[swap]{\beta_{f(2),f(1)}} & & B_{f(2)}
      \arrow{rr}[swap]{\beta_{f(3), f(2)}}& & B_{f(3)}
      \arrow{rr}[swap]{\beta_{f(4),f(3)}}& & \cdots
    \end{tikzcd}
  \]
  with commuting triangles.  A \emph{homotopy morphism} of inductive
  systems is a pair $(f, \underline \phi)$ satisfying the above
  conditions for some sequence of homotopies $\underline h$.  In either case, when $f = \id_{\mathbb N}$, we will drop it from the notation, writing $\underline \phi$ and $(\underline \phi, \underline h)$ in place of $(\mathrm{id}_\mathbb N, \underline \phi)$ and $(\mathrm{id}_\mathbb N, \underline \phi, \underline h)$.
\end{definition}

Dadarlat showed in \cite[Section~2]{Dadarlat94} that a strong homotopy morphism has a \emph{homotopy limit,} defined as an asymptotic morphism on the corresponding inductive limit algebras. We briefly recall the definition of this functor as it is central to this paper. \begin{definition}
  \label{def:homotopy-limit}
  Given a strong homotopy morphism
  \begin{equation*}
    (f, \underline \phi, \underline h) \colon (\underline A, \underline \alpha) \xrightarrow\approx (\underline B, \underline \beta),
  \end{equation*}
  let $A = \varinjlim\,(\underline A, \underline \alpha)$ and $B =
  \varinjlim\, (\underline B, \underline \beta)$.  For $n \in \mathbb N$, define
  $\Phi_n \colon A_n \rightarrow C_b(\mathbb R_{\geq n}, B)$ by
  \begin{equation*}
    \Phi_n(a)(t) = \beta_{\infty, f(m+1)}\Big(
    h_m\big(
    \alpha_{m, n}(a)
    \big)(t - m)
    \Big)
  \end{equation*}
  for $m \geq n$ and $m \leq t < m + 1$.  If $\rho_n \colon
  C_b(\mathbb R_{\geq n}, B) \rightarrow C_b(\mathbb R_{\geq (n+1)})$
  is the restriction map, we have a commuting diagram
  \[
    \begin{tikzcd}
      A_1 \arrow{r}{\alpha_1} \arrow{d}{\Phi_1} & A_2 \arrow{r}{\alpha_2} \arrow{d}{\Phi_2} & A_3 \arrow{r}{\alpha_3} \arrow{d}{\Phi_3} & \hskip 10pt \cdots &[-20pt] A \arrow{d}{\phi_{\rm as}}\\[5pt]
      C_b(\mathbb R_{\geq 1}, B) \arrow{r}{\rho_1} & C_b(\mathbb
      R_{\geq 2}, B) \arrow{r}{\rho_2} & C_b(\mathbb R_{\geq 3}, B)
      \arrow{r}{\rho_3} & \hskip 10pt \cdots & B_{\rm as}
    \end{tikzcd}
  \]
  which induces a $^*$-homomorphism $\phi_{\rm as} \colon A
  \rightarrow B_{\rm as}$.  This, in turn, lifts to an asymptotic
  morphism $\phi \colon A \xrightarrow\approx B$.  We define $\hlim \,
  (f, \underline \phi, \underline h ) = \phi$, which is well-defined
  up to equivalence (since $\phi_{\rm as}$ is well-defined).
\end{definition}

The homotopy limit satisfies the expected compatibility property with the sequence of morphisms defining it.

\begin{proposition}\label{prop:limit-factor}
If $(f, \underline \phi, \underline h) \colon (\underline A, \underline \alpha) \xrightarrow\approx (\underline B, \underline \beta)$ is a strong homotopy morphism of inductive systems with homotopy limit $\phi \colon A \xrightarrow\approx B$, then $[[\phi \circ \alpha_{\infty, n}]] = [[\beta_{\infty, f(n)} \circ \phi_n]]$ for all $n \in \mathbb N$. 	
\end{proposition}

\begin{proof}
	Given $n \in \mathbb N$, define $\Theta_n \colon A_n \rightarrow C_b(\mathbb R_{\geq n}, C([0, 1], B))$ by 
	\begin{equation*}
		\Theta_n(a)(t)(s) = \Phi_n(a)\big(s(t - n) +n\big)
	\end{equation*}
	for $a \in A_n$, $t \in \mathbb R_{\geq n}$, and $s \in [0, 1]$.  The asymptotic morphism $\theta_n \colon A \xrightarrow\approx C([0, 1], B)$ induced by $\Theta_n$ is the desired homotopy.    	
\end{proof}

The following result is due to Dadarlat in \cite[Corollary~3.15]{Dadarlat94}.

\begin{theorem}\label{thm:diagram-rep}
	If $(\underline A, \underline \alpha)$ is a shape system with limit $A$ and $(\underline B, \underline \beta)$ is an inductive system with limit $B$, then every asymptotic morphism $\phi \colon A \xrightarrow\approx B$ has the form $\hlim\, (f, \underline \phi, \underline h)$ for a suitable strong homotopy morphism $(f, \underline \phi, \underline h)$ of the inductive systems.
\end{theorem}

Note that any $C^*$-algebra $B$ is the limit of the inductive system $(\underline B, \underline \beta)$ where $B_n = B$ and $\beta_n = \mathrm{id}_B$ for all $n \in \mathbb N$.  In this case, we will often abuse notation and write $(\underline \phi, \underline h) \colon (\underline A, \underline \alpha) \xrightarrow\approx B$ for a strong homotopy morphism $({\rm id}_\mathbb N, \underline \phi, \underline h)$.

\subsection{Homotopy stability}\label{sec:homotopy-stability}

A separable $C^*$-algebra $A$ is called \emph{semiprojective} if $\mathrm{id}_A$ is semiprojective.  Blackadar proved in \cite[Corollary~4.3]{Blackadar16} that if $\phi, \psi \colon A \rightarrow B$ are $^*$-homomorphisms between $C^*$-algebras and are sufficiently close in the point-norm topology and $A$ is separable and semiprojective, then $\phi$ and $\psi$ are homotopic (with the bound depending only on $A$).  This subsection is devoted to proving a relative version of Blackadar's result for semiprojective $^*$-homomorphisms (Corollary~\ref{cor:homotopy-stability})---the proof is very similar.  We begin with a lifting result similar to \cite[Theorem~4.1]{Blackadar16}.

\begin{theorem}\label{thm:lifting}
	Let $A_0$ and $A$ be separable $C^*$-algebras and let $\alpha \colon A_0 \rightarrow A$ be a semiprojective $^*$-homomorphism.  For every finite set $\mathcal F \subseteq A_0$ and $\epsilon > 0$, there are a finite set $\mathcal G \subseteq A$ and $\delta > 0$ such that for all $C^*$-algebras $B$ and $E$, surjective $^*$-homomorphisms $q \colon E \rightarrow B$, and $^*$-homomorphisms $\phi, \psi \colon A \rightarrow B$ such that \mbox{$\|\phi(a) - \psi(a)\| < \delta$} for all $a \in \mathcal G$, if $\tilde\phi \colon A \rightarrow E$ is a $^*$-homomorphism such that $q\circ  \tilde\phi = \phi$, then there is a $^*$-homomorphism $\tilde\psi \colon A_0\rightarrow E$ such that $q \circ \tilde\psi = \psi \circ \alpha$ and $\|\tilde{\phi}(\alpha(a)) - \tilde{\psi}(a)\| < \epsilon$ for all $a \in \mathcal F$.
\end{theorem}

\begin{proof}
	Suppose the result is false and choose a finite set  $\mathcal F \subseteq A_0$ and $\epsilon > 0$ for which the result fails.  Let $\mathcal G_n \subseteq A$ be an increasing sequence of finite sets with dense union and let $\delta_n > 0$ be a decreasing sequence converging to 0.  For each $n \geq 1$, choose $C^*$-algebras $B_n$ and $E_n$, a surjective $^*$-homomorphism $q_n \colon E_n \rightarrow B_n$, $^*$-homomorphisms $\phi_n, \psi_n \colon A \rightarrow B_n$ with $\|\phi(a) - \psi(a)\| < \delta_n$ for all $a \in \mathcal G_n$ and a $^*$-homomorphism $\tilde\phi_n \colon A \rightarrow E_n$ with $q \circ \tilde \phi_n = \phi_n$ such that for all $^*$-homomorphisms $\tilde \psi_n \colon A_0 \rightarrow E_n$ with $q \circ \tilde \psi_n = \psi_n \circ \alpha$, there is an $a \in \mathcal F$ such that $\|\tilde{\phi}(\alpha(a)) - \tilde{\psi}(a)\| \geq \epsilon$.
	
	Let $\widehat B_n = \prod_{m \geq n} B_m$, $\widehat E_n = \prod_{m \geq n} E_m$, and $\widehat q_n = \prod_{m \geq n} q_m \colon \widehat E_n \rightarrow \widehat B_n$.  Let $\beta_n \colon \widehat B_n \rightarrow \widehat B_{n+1}$ and $\gamma_n \colon \widehat E_n \rightarrow \widehat E_{n+1}$ denote the projection maps.  Let $\widehat B$ and $\widehat E$ denote the inductive limits of $(\underline{\widehat B}, \underline \beta)$ and $(\underline {\widehat E}, \underline \gamma)$, respectively, and let $\widehat q \colon \widehat E \rightarrow \widehat B$ denote the $^*$-homomorphism induced by the $\widehat q_n$. Finally, define $\Phi_n = \prod_{m \geq n} \phi_m$, $\tilde \Phi_n = \prod_{m \geq n} \tilde \phi_m$, and $\Psi_n = \prod_{m \geq n} \psi_m$, and let $\Phi, \Psi \colon A \rightarrow \widehat B$ and $\tilde \Phi \colon A \rightarrow \widehat E$ denote the induced maps.  Note that $\|\phi_n(a) - \psi_n(a) \| \rightarrow 0$ for all $a \in A$, and hence $\Phi = \Psi$.
	
	For $n \geq 1$, consider the pullback 
	\[ P_n = \widehat B_1 \oplus_{\widehat B_n} \widehat E_n = \{(b, e) \in \widehat B_1 \oplus \widehat E_n : \beta_{n,1}(b) = \widehat q_n(e) \}, \]
	and let $\mathrm{pr}_1^{(n)} \colon P_n \rightarrow \widehat B_1$ and $\mathrm{pr}_2^{(n)} \colon P_n \rightarrow \widehat E_n$ denote the projection maps.   Let $J_n$ be the kernel of $\hat{q}_n$ and let $I_n$ be the kernel of $\mathrm{pr}_1^{(n)}$.  There are canonical maps $\theta_n\colon I_n\to J_n$ and $\pi_n\colon P_n\to P_{n+1}$ such that the diagram 
	\newlength{\shift}\setlength{\shift}{10pt}
	\[ \begin{tikzcd}[row sep = normal, column sep = small]
	 &[\shift]0 \arrow{rr} &[-\shift] &[\shift] J_n \arrow{dd} \arrow{rr} &[-\shift] &[\shift] \widehat E_n \arrow{dd}[near end]{\gamma_n}\arrow{rr}{\widehat q_n} &[-\shift] &[\shift] \widehat B_n \arrow{rr} \arrow{dd}[near start]{\beta_n} &[-\shift] &[\shift] 0 \\
	 0 \arrow{rr} & & I_n \arrow[crossing over]{rr} \arrow{ur}[yshift =-3pt, xshift =-1pt]{\theta_n}[swap,xshift = 2pt, yshift = 3pt]{\cong} & & P_n \arrow[crossing over]{rr}{\mathrm{pr}_1^{(n)}} \arrow{ur}[xshift = 2pt, yshift = -4pt]{\mathrm{pr}_2^{(n)}} & & \widehat B_1 \arrow[crossing over]{rr} \arrow{ur}[xshift = -2pt, yshift = -5pt]{\beta_{n, 1}} & & 0 \\
	 & 0 \arrow{rr} & & J_{n+1} \arrow{rr} & & \widehat E_{n+1} \arrow[near end]{rr}[xshift = -3pt]{\widehat q_{n+1}} & & \widehat B_{n+1} \arrow{rr} & & 0 \\
	 0 \arrow{rr}& & I_{n+1} \arrow[crossing over, leftarrow]{uu} \arrow{ur}[xshift = 3pt, yshift = -1pt]{\theta_{n+1}}[swap,xshift =5pt, yshift = 5pt]{\cong} \arrow{rr} & & P_{n+1} \arrow[crossing over, leftarrow]{uu}[swap, near end]{\pi_n} \arrow{ur}[near end, xshift = 3pt, yshift =-6pt]{\mathrm{pr}_2^{(n+1)}} \arrow{rr}[xshift = 8pt]{\mathrm{pr}_1^{(n+1)}} & & \widehat B_1 \arrow[equals, crossing over]{uu} \arrow{ur}[swap,xshift =3pt, yshift = 5pt]{\beta_{n+1, 1}}\arrow{rr} & & 0 & 
	\end{tikzcd} \]
	commutes has exact rows.  A diagram chase shows that each $\theta_n$ is an isomorphism and each $\pi_n$ is surjective.  Taking the inductive limit over $n$ produces a diagram 
	\[ \begin{tikzcd}
		0 \arrow{r} & I \arrow{r} \arrow{d}{\theta}[swap]{\cong} & P \arrow{r}{\mathrm{pr}_1^{(\infty)}} \arrow{d}{\mathrm{pr}_2^{(\infty)}} & \widehat B_1 \arrow{r} \arrow{d}{\beta_{\infty, 1}} & 0 \\
			0 \arrow{r} & J \arrow{r} & \widehat E \arrow{r}{\widehat q} & \widehat B \arrow{r} & 0
	\end{tikzcd} \] 
	that commutes; note that the rows are exact.  A diagram chase shows the right-hand square is a pullback.
	
	The maps $\Psi_1 \colon A \rightarrow \widehat B_1$ and $\tilde\Phi \colon A \rightarrow \widehat E$ induce a $^*$-homomorphism $\rho \colon A \rightarrow P$.  Because $\alpha$ is semiprojective, there are an integer $n \geq 1$ and a $^*$-homomorphism $\tilde \rho \colon A_0 \rightarrow P_n$ such that $\pi_{\infty, n} \circ \tilde \rho = \rho \circ \alpha$.  Write $\mathrm{pr}_2^{(n)} \circ \tilde \rho = \prod_{m \geq n} \tilde \psi_m$ for $^*$-ho\-mo\-mor\-phisms $\tilde\psi_m \colon A_0 \rightarrow E_m$, $m \geq n$.  By construction, $q_m \circ \tilde \psi_m = \psi_m$ for all $m \geq n$, and $\lim_{m \rightarrow \infty} \|\tilde\phi_m(\alpha(a)) - \tilde\psi_m(a)\| = 0$ for all $a \in A_0$, giving a contradiction.
\end{proof}

The following homotopy stability property will be used frequently.  The special case when  $\alpha_0 = \mathrm{id}_A$ is \cite[Corollary~4.3]{Blackadar16}.

\begin{corollary}\label{cor:homotopy-stability}
	Let $A_0$ and $A$ be separable $C^*$-algebras and let $\alpha \colon A_0 \rightarrow A$ be a semiprojective $^*$-homomorphism.  There are a finite set $\mathcal G \subseteq A$ and $\delta > 0$ such that for all $C^*$-algebras $B$ and all $^*$-homomorphisms $\phi, \psi \colon A \rightarrow B$, if $\|\phi(a) - \psi(a)\| < \delta$ for all $a \in \mathcal G$, then $\phi \circ \alpha$ is homotopic to $\psi\circ \alpha$.
\end{corollary}

\begin{proof}
	In the notation of Theorem~\ref{thm:lifting}, let $\mathcal F = \emptyset$ and $\epsilon = 1$, and choose $\mathcal G$ and $\delta$ accordingly.  Consider the surjective $^*$-homomorphism $q \colon C([0, 1], B) \rightarrow B \oplus B$ given by $q(f) = (f(0), f(1))$.  Assume $\phi, \psi\colon A\to B$ are $(\mathcal{G}, \delta)$-close.  The $^*$-homomorphism $\phi \oplus \phi \colon A \rightarrow B \oplus B$ lifts to $E =C([0,1],B)$ and is $(\mathcal{G}, \delta)$-close to $\phi \oplus \psi \colon A \rightarrow B$, so by Theorem~\ref{thm:lifting}, $(\phi \oplus \psi) \circ \alpha$ lifts to a $^*$-homomorphism $A_0 \rightarrow E$.  This lift is the desired homotopy.
\end{proof}

\section{The topology on $[[A, B]]$}
\label{sec:top}

Recall that for $C^*$-algebras $A$ and $B$ with $A$ separable, $[[A,
B]]$ is the set of homotopy classes of asymptotic morphisms from $A$
to $B$.  We will show that there is a topology on $[[A, B]]$ with
properties analogous to those described in Theorem~\ref{thm:main} and
for which the composition product is jointly continuous.  The topology will be defined in Section~\ref{sec:top-def}. Sections~\ref{sec:top-conv} and \ref{sec:top-comp} will address the characterization of convergence and the continuity of composition, respectively.

\subsection{Definition of the topology}\label{sec:top-def}

Before introducing the topology on $[[A, B]]$, we will need a few
relatively standard results that provide homotopy factorization of
asymptotic morphisms through genuine $^*$-homomorphisms in the presence of
semiprojectivity.

\begin{lemma}\label{lem:perturb}
	Let $A_0$, $A$, and $B$ be $C^*$-algebras with $A_0$ and $A$ separable.  If $\alpha \colon A_0 \rightarrow A$ is a semiprojective $^*$-homomorphism and $\phi \colon A \xrightarrow\approx B$ is an asymptotic morphism, then there is an asymptotic morphism $\tilde\phi \colon A_0 \xrightarrow\approx B$ such that $\tilde\phi \cong \phi \circ \alpha$ and $\tilde\phi_t$ is a $^*$-homomorphism for all $t \geq 0$.
\end{lemma}

\begin{proof}
	For $n \geq 1$, let $\rho_n \colon C_b(\mathbb R_{\geq n}, B) \rightarrow C_b(\mathbb R_{\geq (n+1)}, B)$ be the restriction map and identify the limit of this inductive system with $B_{\rm as} = C_b(\mathbb R^+, B) / C_0(\mathbb R_+, B)$.  Since $\alpha_0$ is semiprojective and each $\rho_n$ is surjective, there are an integer $n \geq 1$ and a $^*$-homomorphism $\Phi \colon A_0 \rightarrow C_b(\mathbb R_{\geq n}, B)$ such that $\rho_{\infty, n} \circ \Phi = \phi_{\rm as} \circ \alpha$.  Define $\tilde\phi \colon A \xrightarrow\approx B$ by $\tilde\phi_t(a) = \Phi(a)(t)$ for $t \geq n$ and $\tilde\phi_t(a) = \Phi(a)(n)$ for $0 \leq t < n$.  By construction, $\tilde\phi \cong \phi \circ \alpha$.
\end{proof}

For $C^*$-algebras $A$ and $B$ with $A$ separable, we let $\Homah{A}{B} \subseteq [[A, B]]$ denote the asymptotic homotopy equivalence classes of $^*$-homomorphisms $A \rightarrow B$.

\begin{lemma}\label{lem:factor}
	Let $A_0$, $A$, and $B$ be $C^*$-algebras with $A_0$ and $A$ separable. If $\alpha \colon A_0 \rightarrow A$ is a semiprojective morphism and $\phi \colon A \xrightarrow\approx B$ is an asymptotic morphism, then there is a $^*$-homomorphism $\psi \colon A_0 \rightarrow B$ such that $[[\psi]] = [[\phi \circ \alpha]]$.  In particular, $\alpha$ induces a map
	\[ \alpha^*\colon [[A, B]] \to  \Homah{A_0}{B} \subseteq [[A_0, B]]. \]
\end{lemma}

\begin{proof} 
	Let $\tilde\phi \colon A_0 \xrightarrow\approx B$ be given as in Lemma~\ref{lem:perturb} and set $\psi = \tilde\phi_0$.  For $t \geq 0$, define $\theta_t \colon A \rightarrow C([0, 1], B)$ by $\theta_t(a)(s) = \tilde\phi_{st}(a)$
	for all $a \in A_0$ and $s \in [0, 1]$.  Then $\theta_t$ is an asymptotic homotopy between $\psi$ and $\tilde\phi$.  Hence $[[\psi]] = [[\tilde\phi]] = [[\phi \circ \alpha]]$.
\end{proof}

We will use the maps in Lemma~\ref{lem:factor} to define the topology on $[[A, B]]$.

\begin{definition}\label{def:top-asymp}
	Let $A$ and $B$ be $C^*$-algebras with $A$ separable.
  	\begin{enumerate}
  		\item We equip the set $\Hom(A, B)$ of $^*$-homomorphisms $A \rightarrow B$ with the point-norm topology, so $\phi_n \rightarrow \phi$ if and only if $\|\phi_n(a) - \phi(a)\| \rightarrow 0$ for all $a \in A$.
  		\item We equip the image $\mathrm{H}(A, B)$ of the natural map $\Hom(A, B) \rightarrow [[A, B]]$ with the quotient topology inherited from $\Hom(A, B)$.
		\item  We equip $[[A, B]]$ with the weakest topology such that for every separable $C^*$-algebra $A_0$ and every semiprojective $^*$-homomorphism $\alpha \colon A_0 \rightarrow A$, the map $\alpha^*\colon [[A,B]]\to \Homah{A_0}{B }$ given in  Lemma~\ref{lem:factor} is continuous.
	\end{enumerate}
\end{definition} 

Note that the set $\Homah{A}{B}$ carries two natural topologies:
the quotient topology from $\Hom(A, B)$ and the subspace topology from
$[[A, B]]$.  We will always use the former topology on $\Homah{A}{B}$.
We do not know if these topologies coincide, but they are related via
the following straightforward result.

\begin{proposition}\label{prop:inclusion-continuous}
	If $A$ and $B$ are $C^*$-algebras with $A$ separable, then the inclusion map $\Homah{A}{B} \rightarrow [[A, B]]$ is continuous.
\end{proposition}

\begin{proof}
	It is enough to show that for every separable $C^*$-algebra $A_0$ and semiprojective $^*$-homomorphism $\alpha \colon A_0 \rightarrow A$, the map $\alpha^* \colon \Homah{A}{B} \rightarrow \Homah{A_0}{B}$ is continuous.  This is immediate since the corresponding map $\alpha^* \colon \Hom(A, B) \rightarrow \Hom(A_0, B)$ is continuous.
\end{proof}

The next lemma is stated in \cite[Theorem~2.3]{Thiel19}, where it is noted that
it follows from the closely related results in \cite[Section~3]{Blackadar85}.  We include the details as they are omitted in \cite{Thiel19}.

\begin{lemma}
  \label{lem:sp-homotopy-lift}
  Let $A$ and $B$ be separable $C^*$-algebras and suppose $B$ is the
  limit of an inductive system $(\underline{B}, \underline{\beta})$.  If
  $\alpha\colon A\to B$ is a semiprojective $^*$-homomorphism, then
  there exist an integer $n\geq1$ and a $^*$-homomorphism
  $\tilde{\alpha}\colon A\to B_n$ such that $\beta_{\infty, n}\circ
  \tilde{\alpha}$ is homotopic to $\alpha$.
\end{lemma}

\begin{proof}
  Lemma~\ref{lem:semiproj-factor} implies there are a separable
  $C^*$-algebra $A_1$ and semiprojective $^*$-homomorphisms
  $\alpha_0\colon A\to A_1$ and $\alpha_1\colon A_1\to B$ such that
  $\alpha = \alpha_1\circ \alpha_0$.  Now,
  \mbox{\cite[Theorem~3.1]{Blackadar85}} provides an integer $n\geq 1$ and a
  $^*$-homomorphism $\tilde{\alpha}\colon A\to B_n$ such that
  $\beta_{\infty, n}\circ \tilde{\alpha}$ is homotopic to $\alpha_1  \circ \alpha_0$.
\end{proof}

The following result gives a slightly simpler description of the topology by restricting the collection of semiprojective morphisms that need to be considered in Definition~\ref{def:top-asymp}.

\begin{proposition}\label{prop:top-shape-system}
	If $A$ and $B$ are $C^*$-algebras with $A$ separable and $(\underline{A}, \underline{\alpha})$ is a shape system for $A$, then the topology on $[[A, B]]$ is the weakest topology such that for all $n \geq 1$, the map $\alpha_{\infty, n}^* \colon [[A, B]] \rightarrow \Homah{A_n}{B}$ of Lemma~\ref{lem:factor} is continuous.
\end{proposition}

\begin{proof}
	By the definition of the topology on $[[A, B]]$, each of the maps $\alpha_{\infty, n}^*$ is continuous.  It suffices to show that if $X$ is a topological space, $f \colon X \rightarrow [[A, B]]$, and $\alpha_{\infty, n}^* \circ f$ is continuous for all $n \geq 1$, then $f$ is continuous.  To this end, we must show $\alpha^* \circ f$ is continuous for all separable $C^*$-algebras $A_0$ and semiprojective $^*$-homomorphisms $\alpha \colon A_0\rightarrow A$.  Fix such $A_0$ and $\alpha$.  Using Lemma~\ref{lem:sp-homotopy-lift} and that $\alpha$ is semiprojective, we obtain an integer $n \geq 1$ and a $^*$-homomorphism $\tilde \alpha \colon A_0 \rightarrow A_n$ such that $\alpha_{\infty, n} \circ \tilde\alpha$ is homotopic to $\alpha$.  Then $\alpha^*$ factors as
	\[ [[A, B]] \xrightarrow{\alpha_{\infty, n}^*} \Homah{A_n}{B} \xrightarrow{\tilde \alpha^*} \Homah{A_0}{B}. \]
	Since $\alpha_{\infty, n}^* \circ f$ and $\tilde \alpha^*$ are continuous, so is $\alpha^* \circ f$.
\end{proof}

The following standard result allows us to replace asymptotic homotopies of
$^*$-ho\-mo\-morphisms with genuine homotopies in the presence of
semiprojectivity.  When $\alpha = \mathrm{id}_A$, this follows from \cite[Proposition~25.1.7]{Blackadar98}, for example.  It will be strengthened in Lemma~\ref{lem:t0-equiv-homotopic}.

\begin{lemma}\label{lem:factor2}
	Let $A_0$, $A$, and $B$ be $C^*$-algebras with $A_0$ and $A$ separable and let $\alpha \colon A_0 \rightarrow B$ be a semiprojective $^*$-homomorphism. If $\phi, \psi \colon A \rightarrow B$ are $^*$-ho\-mo\-morphisms with $[[\phi]] = [[\psi]]$, then $\phi \circ \alpha$ and $\psi \circ \alpha$ are homotopic.
\end{lemma}

\begin{proof} 
	Lemma~\ref{lem:semiproj-factor} implies there are a separable $C^*$-algebra $A_1$ and semiprojective $^*$-homomorphisms $\alpha_0 \colon A_0 \rightarrow A_1$ and $\alpha_1 \colon A_1 \rightarrow A$ such that $\alpha = \alpha_1 \circ \alpha_0$.  Use Corollary~\ref{cor:homotopy-stability} to produce a finite set $\mathcal G \subseteq A_1$ and $\delta > 0$ such that if $\phi', \psi' \colon A_1 \rightarrow B$ are $^*$-homomorphisms with $\|\phi'(a) - \psi'(a)\| < \delta$ for all $a \in \mathcal G$, then $\phi' \circ \alpha_0$ and $\psi' \circ \alpha_0$ are homotopic.
	
	By the definition of asymptotic homotopy equivalence and Lemma~\ref{lem:perturb}, there is an asymptotic morphism $\theta \colon A_1 \xrightarrow\approx C([0, 1], B)$ such that $\theta_t$ is a $^*$-homomorphism for all $t \in \mathbb R_+$, $\ev_0 \circ \theta \cong \phi \circ \alpha_1$, and $\ev_1 \circ \theta \cong \psi \circ \alpha_1$.  For some sufficiently large $t_0 \in \mathbb R_+$, we have $\|\theta_{t_0}(a)(0) - \phi(\alpha_1(a))\| < \delta$ and $\|\theta_{t_0}(a)(1) - \psi(\alpha_1(a)) \| < \delta$ for all $a \in \mathcal G$.  By the choice of $\mathcal G$ and $\delta$, $\phi \circ \alpha$ is homotopic to $\ev_0 \circ \theta_{t_0} \circ \alpha_0$ and $\psi \circ \alpha$ is homotopic to $\ev_1 \circ \theta_{t_0} \circ \alpha_0$.  Since $\ev_0 \circ \theta_{t_0} \circ \alpha_0$ and $\ev_1 \circ \theta_{t_0} \circ \alpha_0$ are homotopic, so are $\phi \circ \alpha$ and $\psi \circ \alpha$.
\end{proof}

The following result and its corollary below will allow us to deduce structural properties of the topology on asymptotic morphisms.  

\begin{proposition}\label{prop:discrete-factor}
	Let $A_0$, $A$, and $B$ be $C^*$-algebras with $A_0$ and $A$ separable and let $\alpha \colon A_0 \rightarrow A$ be a semiprojecive $^*$-homomorphisms.   The natural map
	\[ \alpha^* \colon \Homah{A}{B} \rightarrow \Homah{A_0}{B} \]
	factors through a discrete topological space.  Moreover, if $B$ is separable, then this space may be taken to be countable.
\end{proposition}

\begin{proof}
	By Lemma~\ref{lem:semiproj-factor}, there are a separable $C^*$-algebra $A_1$ and semiprojective $^*$-homomorphisms $\alpha_0 \colon A_0 \rightarrow A_1$ and $\alpha_1 \colon A_1 \rightarrow A$ such that $\alpha = \alpha_1 \circ \alpha_0$.  By Corollary~\ref{cor:homotopy-stability}, there are finite set $\mathcal G \subseteq A_1$ and $\delta > 0$ such that for all $^*$-homomorphisms $\phi, \psi \colon A_1 \rightarrow B$, if $\|\phi(a) - \psi(a)\| < \delta$ for all $a \in \mathcal G$, then $\phi \circ \alpha_0$ and $\psi \circ \alpha_0$ are homotopic.
	
	Let $\sim$ be the equivalence relation on $\Hom(A_1, B)$ generated by declaring $\phi \sim \psi$ if $\|\phi(a) - \psi(a)\| < \delta$ for all $a \in \mathcal G$.  Let $X = \Hom(A_1, B) / \sim$ be the quotient space and let $q \colon \Hom(A_1, B) \rightarrow X$ be the quotient map.  For all $x \in X$, the set $q^{-1}(x) \subseteq \Hom(A_1, B)$ is open, and hence $X$ is discrete.  It follows immediately that $q$ is open.  When $B$ is separable, $\Hom(A, B)$ is second countable (being a separable metrizable space), so $X$ is also second countable and therefore countable.
	
    If $\phi, \psi \colon A_1 \rightarrow B$ are $^*$-homomorphisms with $\phi \sim \psi$, then $\phi \circ \alpha_0$ and $\psi \circ \alpha_0$ are homotopic by the choice of $\mathcal G$ and $\delta$.  Therefore, there is a continuous map $f \colon X \rightarrow \Homah{A_0}{B}$ by $f(q(\phi)) = [[\phi \circ \alpha_0]]$ for all $^*$-homomorphisms $\phi \colon A_1 \rightarrow B$.  Now suppose $\phi, \psi \colon A \rightarrow B$ are $^*$-homomorphisms with $[[\phi]] = [[\psi]]$.  By Lemma~\ref{lem:factor2}, $\phi \circ \alpha_1$ and $\psi \circ \alpha_1$ are homotopic, and in particular, $\phi \circ \alpha_1 \sim \psi \circ \alpha_1$.  It follows that there is a continuous map $g \colon \Homah{A}{B} \rightarrow X$ defined by $g([[\phi]]) = q(\phi \circ \alpha_1)$ for all $^*$-homomorphisms $\phi \colon A \rightarrow B$.  By construction, for all $^*$-homomorphisms $\phi \colon A \rightarrow B$, we have 
    \[ f(g([[\phi]])) = f(q(\phi \circ \alpha_1)) = [[\phi \circ \alpha_1 \circ \alpha_0]] = [[\phi \circ \alpha]], \]
    and thus $f \circ g = \alpha^*$.
\end{proof}

\begin{corollary}\label{cor:discrete-factor}
	Let $A_0$, $A$, and $B$ be $C^*$-algebras with $A_0$ and $A$ separable and let $\alpha \colon A_0 \rightarrow A$ be semiprojecive.   The map
	\[ \alpha^* \colon [[A, B]] \rightarrow \Homah{A_0}{B} \]
	of Lemma~\ref{lem:factor}
	factors through a discrete topological space.  Moreover, if $B$ is separable, then this space may be taken to be countable.
\end{corollary}

\begin{proof}
	Using Lemma~\ref{lem:semiproj-factor}, there are a separable $C^*$-algebra $A_1$ and semiprojective $^*$-homomorphisms $\alpha_0 \colon A_0 \rightarrow A_1$ and $\alpha_1 \colon A_1 \rightarrow A$ such that $\alpha = \alpha_1 \circ \alpha_0$. Then $\alpha^*$ factors as
	\[ [[A, B]] \xrightarrow{\alpha_1^*} \Homah{A_1}{B} \xrightarrow{\alpha_0^*} \Homah{A_0}{B}. \]
	The result follows from Proposition~\ref{prop:discrete-factor} applied to $\alpha_0$.
\end{proof}

The factorization result in the previous corollary allows us to prove countability axioms for the space $[[A, B]]$.

\begin{theorem}\label{thm:countability}
	If $A$ and $B$ are $C^*$-algebras with $A$ separable, then $[[A, B]]$ is first countable.  If, in addition, $B$ is separable, then $[[A, B]]$ is second countable.
\end{theorem}

\begin{proof}
	Fix a shape system $(\underline A, \underline \alpha)$ for $A$.  By Corollary~\ref{cor:discrete-factor}, for each $n \in \mathbb N$, there are a discrete space $X_n$ and continuous maps $g_n \colon [[A, B]] \rightarrow X_n$ and $f_n \colon X_n \rightarrow \Homah{A_n}{B}$ such that $\alpha_{\infty, n}^* = f_n \circ g_n$.  Let $\phi \colon A \xrightarrow\approx B$ be an asymptotic morphism and let $U_n = g_n^{-1}(g_n([[\phi]]))$ for all $n \in \mathbb N$.  We will show the sets $\{U_n : n \in \mathbb N \}$ form a neighborhood basis for $[[\phi]]$ in $[[A, B]]$.
	
	For all $n \in \mathbb N$, the set $U_n$ is open as $g_n$ is continuous and $X_n$ is discrete.  By Proposition~\ref{prop:top-shape-system}, open sets of the form $(\alpha_{\infty, n}^*)^{-1}(V_n)$, for an open set $V_n \subseteq \Homah{A_n}{B}$ and $n \in \mathbb N$, form a basis for $[[A, B]]$.  Fix such an open set $(\alpha_{\infty, n}^*)^{-1}(V_n)$ containing $[[\phi]]$.  Then $g_n([[\phi]]) \in f_n^{-1}(V_n)$, and hence $U_n \subseteq (\alpha_{\infty, n}^*)^{-1}(V_n)$, which proves the claim.
	
	When $B$ is separable, we may take each $X_n$ to be countable by Corollary~\ref{cor:discrete-factor}.  Then
	\[ \bigcup_{n=1}^\infty \{ g_n^{-1}(x_n) : x_n \in X_n \} \]
	is a countable collection of open subsets of $[[A, B]]$ containing a neighborhood basis for each point in $[[A, B]]$.  Hence $[[A, B]]$ is second countable.  
\end{proof}

\subsection{Convergence in $[[A, B]]$}\label{sec:top-conv}

Since the topology on $[[A, B]]$ is first countable (Theorem~\ref{thm:countability}), it is determined by its convergent sequences.  This subsection provides a characterization of sequential convergence in $[[A, B]]$ analogous to the characterization of convergence in $E(A, B)$ stated in Theorem~\ref{thm:main} and the one of convergence in $KK(A, B)$ given in \cite[Theorem~3.5]{Dadarlat05}. We begin with the following weak continuity result for composition in $[[A, B]]$.  We will later show in Theorem~\ref{thm:joint-continuity} that composition is in fact jointly continuous.

\begin{lemma}\label{lem:asymp-convergence}
	Let $A$, $B$, and $D$ be $C^*$-algebras with $A$ separable.  If $(\phi_n)_{n=1}^\infty$ is a sequence in $\Hom(B, D)$ converging to $\phi$ and $x \in [[A, B]]$, then $[[\phi_n]] \circ x \rightarrow [[\phi]] \circ x$ in $[[A, D]]$.
\end{lemma}

\begin{proof}
	It suffices to show that for each semiprojective $^*$-homomorphism $\alpha \colon A_0 \rightarrow A$, we have $\alpha^*([[\phi_n]] \circ x) \rightarrow \alpha^*([[\phi]] \circ x)$ in $\Homah{A_0}{D}$.  Note that for each $n \in \mathbb N$, $\alpha^*([[\phi_n]] \circ x) = [[\phi_n]] \circ \alpha^*(x)$ and $\alpha^*([[\phi]] \circ x) = [[\phi]] \circ \alpha^*(x)$.  By Lemma~\ref{lem:factor}, there is a $^*$-homomorphism $\psi \colon A_0 \rightarrow B$ such that $\alpha^*(x) = [[\psi]]$.  Since $\phi_n \circ \psi \rightarrow \phi \circ \psi$ in $\Hom(A_0, D)$, the result follows.  
\end{proof}

Write $\mathbb N$ for the natural numbers equipped with the discrete topology and let $\mathbb N^\dag = \mathbb N \cup \{\infty\}$ be the one-point compactification of $\mathbb N$.  For $C^*$-algebras $A$ and $B$ with $A$ separable, $y \in [[A, C(\mathbb N^\dag, B)]]$, and $m \in \mathbb N^\dag$, define $y(m) = [[\ev_m]] \circ y \in [[A, B]]$.

\begin{theorem}[Pimsner's Condition]\label{thm:asymp-convergence}
	Suppose $A$ and $B$ be $C^*$-algebras with $A$ separable, $(x_m)_{m=1}^\infty$ is a sequence in $[[A, B]]$, and $x \in [[A, B]]$.  Then $x_m \rightarrow x$ in $[[A, B]]$ if and only if there exists $y \in [[A, C(\mathbb N^\dag, B)]]$ such that $y(m) = x_m$ for all $m \in \mathbb N$ and $y(\infty) = x$.
\end{theorem}

\begin{proof}
	First suppose $y \in [[A, C(\mathbb N^\dag, B)]]$ satisfies $y(m) = x_m$ for all $m \in \mathbb N$ and $y(\infty) = x$.  Since $\ev_m \rightarrow \ev_\infty$ in $\Hom(C(\mathbb N^\dag, B), B)$, Lemma~\ref{lem:asymp-convergence} implies that $y(m) \rightarrow y(\infty)$.

	Conversely, suppose $x_m \rightarrow x \in [[A, B]]$.  Fix a shape system $(\underline{A}, \underline{\alpha})$ for $A$.  Represent $x_m$ and $x$ by strong homotopy morphisms $(\underline{\phi}^{(m)}, \underline{h}^{(m)})$ and $(\underline{\phi}, \underline{h})$ from $(\underline A, \underline \alpha)$ to $B$ as in Theorem~\ref{thm:diagram-rep}, where we are regarding $B$ as the limit of a constant inductive system as in the remarks following Theorem~\ref{thm:diagram-rep}.  Write $\phi^{(m)}$ and $\phi$ for the homotopy limits of  $(\underline{\phi}^{(m)}, \underline{h}^{(m)})$ and $(\underline{\phi}, \underline{h})$, respectively. Proposition~\ref{prop:limit-factor} implies $[[\phi_n^{(m)}]] = [[\phi^{(m)} \circ \alpha_{\infty, n}]]$ and $[[\phi_n]] = [[\phi \circ \alpha_{\infty, n}]]$ for all $m,n \in \mathbb N$.  By Corollary~\ref{cor:discrete-factor}, the map
	\[ \alpha_{\infty, n+1}^* \colon [[A, B]] \rightarrow \Homah{A_{n+1}}{B} \]
	factors through a discrete space. Therefore, for all $n \in \mathbb N$, there is an $m_n \in \mathbb N$ such that for all $m > m_n$, we have $[[\phi_{n+1}^{(m)}]] = [[\phi_{n+1}]]$.  The semiprojectivity of $\alpha_n$ and  Lemma~\ref{lem:factor2} imply that $\phi_n^{(m)}$ is homotopic to $\phi_n$ for all $m > m_n$.  Enlarging $m_n$ if necessary, we may assume the sequence $(m_n)_{n=1}^\infty$ is strictly increasing.
	
	For $m, n \in \mathbb N$ with $m > m_n$, we have $\phi_n$ is homotopic to $\phi_n^{(m)}$, and $\phi_n^{(m)}$ is homotopic to $\phi_{n+1}^{(m)} \circ \alpha_n$.  Let $k_n^{(m)} \colon A_n \rightarrow C([0, 1], B)$ be a homotopy from $\phi_n$ to $\phi_{n+1}^{(m)} \circ \alpha_n$.  For $n \in \mathbb N$, define $\Phi_n \colon A_n \rightarrow C(\mathbb N^\dag, B)$ by
	\begin{align*} 
		\Phi_n(a)(m) &= 
		\begin{cases} 
			\phi_n^{(m)}(a), & 1 \leq m \leq m_n \\
			\phi_n(a), & m_n < m \leq \infty
		\end{cases}
	\intertext{for all $a \in A$ and $m \in \mathbb N^\dag$. Further, define $H_n \colon A_n \rightarrow C([0, 1], C(\mathbb N^\dag, B))$ by}
		H_n(a)(s)(m) &= 
		\begin{cases}
			h_n^{(m)}(a)(s), & 1 \leq m \leq m_n \\
			k_n^{(m)}(a)(s), & m_n < m \leq m_{n+1} \\
			h_n(a)(s), & m_{n+1} < m \leq \infty
		\end{cases}
	\end{align*}
	for all $a \in A$, $s \in [0,1]$, and $m \in \mathbb N^\dag$.  Then $(\underline \Phi, \underline H) \colon (\underline A, \alpha) \rightarrow C(\mathbb N^\dag, B)$ defines a strong homotopy morphism.  If $y = \hlim\, (\underline \Phi, \underline H)$, we have $y(m) = x_m$ for $m \in \mathbb N$ and $y(\infty) = x$.
\end{proof}

\subsection{Continuity of composition}\label{sec:top-comp}

We will use the characterization of convergent sequences in Theorem~\ref{thm:asymp-convergence} to strengthen the continuity result from Lemma~\ref{lem:asymp-convergence} to the one in Theorem~\ref{thm:joint-continuity}.  First we record two preliminary results.  

The following proposition is routine.  For $C^*$-algebras $A$ and $D$, we write $A \otimes D$ for the maximal tensor product of $A$ and $D$.

 \begin{proposition}\label{prop:tensor}
	If $A$, $B$, and $D$ are $C^*$-algebras and $\phi \colon A \xrightarrow\approx B$ is an asymptotic morphism, then there is an asymptotic morphism $\mathrm{id}_D \otimes \phi \colon D \otimes A \xrightarrow\approx D \otimes B$, unique up to equivalence, that is determined by
	\begin{equation}
		\lim_{t \rightarrow \infty} \|(\mathrm{id}_D \otimes \phi)_t(d \otimes a) - d \otimes \phi_t(a) \| = 0
	\end{equation}
	for all $a \in A$ and $d \in D$.  Moreover, the assignment $\phi \mapsto \mathrm{id}_D \otimes \phi$ is natural in the sense that if $D_1$ and $D_2$ are $C^*$-algebras and $\theta \colon D_1 \rightarrow D_2$ is a $^*$-homomorphism, then $(\theta \otimes \mathrm{id}_B) \circ (\mathrm{id}_{D_1} \otimes \phi) \cong (\mathrm{id}_{D_2} \otimes \phi) \circ (\theta \otimes \mathrm{id}_{A})$. 
\end{proposition}

\begin{proof}
	For uniqueness, note that if $\psi, \psi'\colon D \otimes A \xrightarrow\approx D \otimes B$ satisfy
	\[ \lim_{t \rightarrow 0} \|\psi(d \otimes a) - d \otimes \phi(a) \| = \lim_{t \rightarrow 0} \|\psi'(d \otimes a) - d \otimes \phi(a) \| = 0 \]
	for all $a \in A$ and $d \in D$, then $\|\psi_t(c) - \psi'_t(c)\| \rightarrow 0$ for all $c$ in the algebraic tensor product of $D$ and $A$.  Using the asymptotic contractivity of $\psi$ and $\psi'$, this also holds for all $c \in D \otimes A$, so $\psi \cong \psi'$.
	
	For existence, define $\rho \colon D \otimes C_b(\mathbb R_+, B) \rightarrow C_b(\mathbb R_+, D \otimes B)$ by $\rho(d \otimes f)(t) = d \otimes f(t)$.  Then $\rho$ restricts to an isomorphism from $D \otimes C_0(\mathbb R_+, B)$ to $C_0(\mathbb R_+, D \otimes B)$, so $\rho$ induces a $^*$-homomorphism $\bar\rho \colon D \otimes B_{\rm as} \rightarrow (D \otimes B)_{\rm as}$.  Let $\mathrm{id}_D \otimes \phi$ be an asymptotic morphism lifting $\bar\rho \circ (\mathrm{id}_D \otimes \phi_{\rm as})$.  The naturality follows from the naturality of tensor products and $\rho$.
\end{proof}

Specializing to the case when $D$ is commutative gives the following result.  It is not clear to us whether the map $\bar\phi$ below is unique (up to equivalence), but only the existence of such a map $\bar\phi$ will be needed.

\begin{corollary}\label{cor:tensor}
	If $A$ and $B$ are $C^*$-algebras, $X$ is a locally compact Hausdorff space, and $\phi \colon A \xrightarrow\approx B$ is an asymptotic morphism, then there is an asymptotic morphism $\bar\phi \colon C_0(X, A) \xrightarrow\approx C_0(X, B)$ such that $\ev_x \circ \bar\phi \cong \phi \circ \ev_x$ for all $x \in X$.
\end{corollary}

\begin{proof}
Define isomorphisms 
\[ \psi_A \colon C_0(X) \otimes A \rightarrow C_0(X, A)\quad \text{and} \quad \psi_B \colon C_0(X) \otimes B \rightarrow C_0(X, B) \]
by $\psi_A(f \otimes a)(x) = f(x)a$ and $\psi_B(f \otimes b)(x) = f(x)b$ for $f \in C_0(X)$, $a \in A$, $b \in B$, and $x \in X$.  Then define 
\[ \bar\phi = \psi_B \circ (\mathrm{id}_{C_0(X)} \otimes \phi) \circ \psi_A^{-1} \colon C_0(X, A) \xrightarrow\approx C_0(X, B). \]
The result follows from the naturality of the tensor product in Proposition~\ref{prop:tensor} applied to $\theta = \ev_x \colon C_0(X) \rightarrow \mathbb C$.
\end{proof}

Finally, we prove the joint continuity of composition.

\begin{theorem}\label{thm:joint-continuity}
	If $A$, $B$, and $D$ are $C^*$-algebras with $A$ and $B$ separable, then the composition $[[A, B]] \times [[B, D]] \rightarrow [[A, D]]$ is jointly continuous.
\end{theorem}

\begin{proof}
	By Theorem~\ref{thm:countability}, both $[[A, B]]$ and $[[B, D]]$ are first countable, and hence so is $[[A, B]] \times [[B, D]]$.  So it suffices to show composition is sequentially continuous.  Suppose $x_n \rightarrow x$ in $[[A, B]]$ and $y_n \rightarrow y$ in $[[B, D]]$.  By Theorem~\ref{thm:asymp-convergence}, there are asymptotic morphisms $\phi \colon A \xrightarrow\approx C(\mathbb N^\dag, B)$ and $\psi \colon B\xrightarrow\approx C(\mathbb N^\dag, D)$ such that
	\[
		[[\ev_n \circ \phi]] = \begin{cases} x_n & n < \infty \\ x & n = \infty \end{cases} \qquad \text{and} \qquad [[\ev_n \circ \psi]] = \begin{cases} y_n & n < \infty \\ y & n = \infty \end{cases}
	\]
	for all $n \in \mathbb N^\dag$.
	
	After identifying $C(\mathbb N^\dag, C(\mathbb N^\dag, D))$ with $C(\mathbb N^\dag \times \mathbb N^\dag, D)$, Corollary~\ref{cor:tensor} provides an asymptotic morphism $\bar\psi \colon C(\mathbb N^\dag, B) \rightarrow C(\mathbb N^\dag \times \mathbb N^\dag, D)$ such that 
	\[ \ev_{m, n}\circ \bar\psi \cong  \ev_m \circ \psi \circ \ev_n \]
	for all $m, n \in \mathbb N^\dag$.  Let $z = [[\bar\psi]] \circ [[\phi]] \in [[A, C(\mathbb N^\dag\times \mathbb N^\dag, D)]]$.  Then 
	\[ [[\ev_{n, n}]] \circ z = \begin{cases} y_n \circ x_n & n < \infty \\ y \circ x & n = \infty \end{cases} \]
	for $n \in \mathbb N^\dag$.  Since $\ev_{n, n} \rightarrow \ev_{\infty, \infty}$ in $\Hom(C(\mathbb N^\dag \times \mathbb N^\dag, D), D)$, Lemma~\ref{lem:asymp-convergence} implies $y_n \circ x_n \rightarrow y \circ x$.
\end{proof}

\section{Hausdorffized asymptotic morphisms}\label{sec:Hd}

The topology on $[[A, B]]$ given in Definition~\ref{def:top-asymp} is often non-Hausdorff.  In Section~\ref{sec:Hd-cat}, we consider a quotient of $[[A, B]]$ that is Hausdorff and show that the composition descends to a continuous composition on the quotient.  Further properties of the quotient are established in Section~\ref{sec:basic-properties-ash}, and a compatibility result with inductive limits in the spirit of Theorem~\ref{thm:cont-EL} is given in Section~\ref{sec:limits}.

\subsection{The Hausdorffized asymptotic category}\label{sec:Hd-cat}

Every topological space $X$ admits a universal $T_0$ quotient space known as the \emph{Kolmogorov quotient}.  While this is typically non-Hausdorff, we will show that the Kolmogorov quotient of the space $[[A, B]]$ of asymptotic morphisms is always a Hausdorff space (Theorem~\ref{thm:hausdorff}).  Elements of this quotient (defined formally below) will be the morphisms of the Hausdorffized asymptotic category.

\begin{definition}
  \label{def:haus-asym}
	Let $A$ and $B$ be $C^*$-algebras with $A$ separable.  For $x, y \in [[A, B]]$, write $x \sim_{\mathrm{Hd}} y$ in $[[A, B]]$ if the singletons $\{x\}$ and $\{y\}$ have the same closure.  Then $\sim_{\mathrm{Hd}}$ is an equivalence relation.  We define the space of \emph{Hausdorffized asymptotic morphisms}, written $[[A, B]]_{\mathrm{Hd}}$, to be the quotient space $[[A, B]] /\!\sim_{\mathrm{Hd}}$.
\end{definition}

Note that the quotient map $[[A, B]] \rightarrow [[A, B]]_{\mathrm{Hd}}$ induces a bijection on open sets, and in particular, the quotient map is open.  As we now show, this observation implies that the composition of asymptotic morphisms descends to a composition on the quotient spaces.

\begin{proposition}\label{prop:composition}
	If $A$, $B$, and $D$ are $C^*$-algebras such that $A$ and $B$ separable, then the  composition $[[A, B]] \times [[B, D]] \rightarrow [[A, D]]$ induces a continuous map $[[A, B]]_{\mathrm{Hd}} \times [[B, D]]_{\mathrm{Hd}} \rightarrow [[A, D]]_{\mathrm{Hd}}$, written $(x, y) \mapsto y \circ x$.
\end{proposition}

\begin{proof}
	Consider the quotient maps
	\[ q \colon [[A, B]] \rightarrow [[A, B]]_{\mathrm{Hd}},\ q' \colon [[B, D]] \rightarrow [[B, D]]_{\mathrm{Hd}},\ \text{and}\ q'' \colon [[A, D]] \rightarrow [[A, D]]_{\mathrm{Hd}}. \]  If $x, y \in [[A, B]]$ and $x', y' \in [[B, D]]$ with $q(x) = q(y)$ and $q'(x') = q'(y')$, then the constant sequences $x$ and $x'$ converge to $y$ and $y'$, respectively.  By the continuity of composition (Theorem~\ref{thm:joint-continuity}), the constant sequence $x' \circ x$ converges to $y' \circ y$.  Similarly, the constant sequence $y' \circ y$ converges to $x' \circ x$.  These convergences imply $q''(x' \circ x) = q''(y' \circ y)$.  Therefore, the composition on the Kolmogorov quotients is well-defined.  Since $q$ and $q'$ are open, so is $q \times q'$.  Hence $q \times q'$ is a quotient map, and the continuity of composition on the Kolmogorov quotients follows.
\end{proof}	

\begin{definition}
  \label{def:hd-asym-cat}
	The \emph{Hausdorffized asymptotic category} $\AMH$ is the category with objects given by separable $C^*$-algebras, the morphisms from $A$ to $B$ given by the set $[[A, B]]_{\mathrm{Hd}}$, and the composition given by Proposition~\ref{prop:composition}.
\end{definition}

In Section~\ref{sec:applications}, we prove that the category $\AMH$ is equivalent to the \emph{shape category} considered in \cite{Dadarlat94} (Theorem~\ref{thm:haus-functor-equiv}).  As a consequence, isomorphism in the category $\AMH$ (and in $\AM$) coincides with shape equivalence of separable $C^*$-algebras.

The following lemma gives a strengthening of Lemma~\ref{lem:factor2}, weakening the hypothesis from agreement in $[[A, B]]$ to agreement in $[[A, B]]_{\mathrm Hd}$.  We record it here for use in Section~\ref{sec:applications}.

\begin{lemma}
	\label{lem:t0-equiv-homotopic}
	Let $A_0$, $A$, and $B$ be $C^*$-algebras with $A_0$ and $A$ separable and let
	$\alpha\colon A_0\to A$ be a semiprojective $^*$-homomorphism.  If
	$\phi, \psi\colon A\to B$ are $^*$-ho\-mo\-mor\-phisms with $[[\phi]]_{\mathrm{Hd}}
	= [[\psi]]_{\mathrm{Hd}}$, then $\phi\circ \alpha$ and $\psi\circ \alpha$
	are homotopic.
\end{lemma}

\begin{proof}
	By Lemma~\ref{lem:semiproj-factor}, there are separable
	$C^*$-algebras $A_1$ and $A_2$ and semiprojective $^*$-homomorphisms
	$\alpha_0\colon A_0\to A_1$, $\alpha_1\colon A_1\to A_2$ and
	$\alpha_2\colon A_2\to A$ such that $\alpha = \alpha_2 \circ
	\alpha_1 \circ \alpha_0$.
	Let $\mathcal{G} \subseteq A_2$ and $\delta > 0$ be given by applying
	Corollary~\ref{cor:homotopy-stability} to $\alpha_1$. 
	
	From $[[\phi]]_{\mathrm{Hd}} = [[\psi]]_{\mathrm{Hd}}$ and
	Theorem~\ref{thm:asymp-convergence}, there is an asymptotic morphism
	$\eta\colon A \xrightarrow{\approx} C(\mathbb{N}^\dag , B)$ with
	$[[\ev_\infty \circ \eta]] = [[\psi]]$ and
	$[[\ev_n \circ \eta]] = [[\phi]]$ for all $n\in
	\mathbb{N}$.
	By Lemma~\ref{lem:factor}, there exists a $^*$-homomorphism
	$\tilde{\eta}\colon A_2\to C(\mathbb{N}^\dag , B)$ with
	$[[\tilde{\eta}]] = [[\eta\circ \alpha_2]]$.  Therefore,
	$[[\ev_\infty\circ \tilde{\eta}]] = [[\psi \circ \alpha_2]]$ and
	$[[\ev_n\circ \tilde{\eta}]] = [[\phi\circ \alpha_2]]$ for all $n\in
	\mathbb{N}$.
	Choose $N\in \mathbb{N}$ so large that $\|(\ev_N \circ \tilde{\eta})(a) -
	(\ev_\infty\circ \tilde{\eta})(a) \| < \delta$ for all $a\in
	\mathcal{G}$.  The choice of $\mathcal{G}$ and $\delta$ implies that
	$\ev_N\circ \tilde{\eta} \circ \alpha_1$ is homotopic to
	$\ev_\infty\circ \tilde{\eta}\circ \alpha_1$.  Therefore,
	$[[\phi\circ \alpha_2 \circ \alpha_1]] = [[\psi\circ \alpha_2 \circ
	\alpha_1]]$.  Finally, Lemma~\ref{lem:factor2} implies that $\phi\circ\alpha_2\circ
	\alpha_1\circ \alpha_0$ is homotopic to $\psi\circ \alpha_2 \circ
	\alpha_1 \circ \alpha_0$.
\end{proof}

\subsection{Basic properties of $[[A,B]]_{\mathrm{Hd}}$}
\label{sec:basic-properties-ash}

Shape systems and strong homotopy morphisms have proved useful in the study of asymptotic morphisms.  They will be equally powerful in the study of the Hausdorffized asymptotic category.  The following gives a formal statement of the claim made in the remarks following Theorem~\ref{thm:haus-shape} that Dadarlat's homotopy limit functor is independent of the choice of homotopies up to Hausdorffization.

\begin{proposition}\label{prop:homotopy-limit}
	Let $(\underline A, \underline \alpha)$ and $(\underline B, \underline \beta)$ be inductive systems of $C^*$-algebras with limits $A$ and $B$ and assume each $A_n$ is separable.  Let 
	\[ (f, \underline \phi, \underline h^\phi), (g, \underline \psi, \underline h^\psi) \colon (\underline A, \underline \alpha) \rightarrow (\underline B, \underline \beta) \]
	be strong homotopy morphisms with homotopy limits $\phi, \psi \colon A \xrightarrow\approx B$.   If $\beta_{\infty, f(n)} \circ \phi_n$ and $\beta_{\infty, g(n)} \circ \psi_n$ are homotopic for all $n \in \mathbb N$, then $[[\phi]]_{\mathrm{Hd}} = [[\psi]]_{\mathrm{Hd}}$.
\end{proposition}

\begin{proof}
	After replacing the $^*$-homomorphisms $\phi_n$ and $\psi_n$ with $\beta_{\infty, f(n)} \circ \phi_n$ and $\beta_{\infty, g(n)}\circ \psi_n$ and the homotopies $h_n^\phi$ and $h_n^\psi$ with $(\mathrm{id}_{C([0,1])} \otimes \beta_{\infty, f(n+1)}) \circ h^\phi_n$ and $(\mathrm{id}_{C([0,1])} \otimes \beta_{\infty, g(n+1)}) \circ h_n^\psi$, we may assume that $B_n = B$ and $\beta_n = \mathrm{id}_B$ for all $n \in \mathbb N$ and $f = g = \mathrm{id}_\mathbb N$.  Hence, after changing notation, we have strong homotopy morphisms $(\underline \phi, \underline h^\phi), (\underline \psi, \underline h^\psi) \colon (\underline A, \underline \alpha) \rightarrow B$ with homotopy limits $\phi, \psi \colon A \xrightarrow\approx B$ such that $\phi_n$ and $\psi_n$ are homotopic for all $n \in \mathbb N$.
	
	For $n \in \mathbb N$, we have $\phi_n$ is homotopic to both $\psi_n$ and $\phi_{n+1} \circ \alpha_n$.  Consider a homotopy $k_n \colon A_n \rightarrow C([0, 1], B)$ from $\psi_n$ to $\phi_{n+1} \circ \alpha_n$.  Define a strong homotopy morphism $(\underline \theta, \underline l) \colon (\underline A, \underline \alpha) \rightarrow C(\mathbb N^\dag, B)$ by
	\begin{align*}
		\theta_n(a)(m) &= 
		\begin{cases}
			\phi_n(a) & m < n \\
			\psi_n(a) & m \geq n
		\end{cases}
	\intertext{and}
		l_n(a)(m)(s) &= 
		\begin{cases} 
			h_n^\phi(a)(s) & m < n \\
			k_n(a)(s) & m = n \\
			h_n^\psi(a)(s) & m > n
		\end{cases}
	\end{align*}
	for all $n \in \mathbb N$, $m \in \mathbb N^\dag$, $a \in A$, and $s \in [0, 1]$.  Let $\theta = \hlim\, (\underline \theta, \underline l) \colon A \xrightarrow{\approx} B$ and note that $[[\mathrm{ev}_m \circ \theta]] = [[\phi]]$ for $m \in \mathbb N$ and $[[\mathrm{ev}_\infty \circ \theta]] = [[\psi]]$.  By Theorem~\ref{thm:asymp-convergence}, it follows that $[[\psi]]$ belongs to the closure of $\{[[\phi]]\}$.  By symmetry, $[[\phi]]$ belongs to the closure of $\{[[\psi]]\}$, and hence $[[\phi]]_{\mathrm{Hd}} = [[\psi]]_{\mathrm{Hd}}$.  
\end{proof}

The following result will be used to produce a projective limit decomposition of the space $[[A, B]]_{\rm Hd}$ in the proof of Theorem~\ref{thm:hausdorff}, which in particular, will be used to prove $[[A, B]]_{\rm Hd}$ is Hausdorff.

\begin{lemma}\label{lem:seq-pt-wise}
	If $A_0$, $A$, and $B$ are $C^*$-algebras with $A_0$ and $A$ separable, \mbox{$\alpha \colon A_0 \rightarrow A$} is a semiprojective $^*$-homomorphism, and $x, y \in [[A,C(\mathbb N^\dag, B)]]$ with \mbox{$x(m) = y(m)$} for all $m \in \mathbb N^\dag$, then $x \circ [[\alpha]] = y \circ [[\alpha]]$.
\end{lemma}

\begin{proof}
	By two applications of Lemma~\ref{lem:semiproj-factor}, we may factor $\alpha$ as a composition
	\[ \begin{tikzcd}
		A_0 \arrow{r}{\alpha_0} & A_1 \arrow{r}{\alpha_1} & A_2 \arrow{r}{\alpha_2} & A 
	\end{tikzcd} \]
	where $A_1$ and $A_2$ are separable $C^*$-algebras and $\alpha_0$, $\alpha_1$, and $\alpha_2$ are semiprojective \mbox{$^*$-homomorphisms}.  Lemma~\ref{lem:factor}, provides $^*$-homomorphisms $\theta, \rho \colon A_2 \rightarrow C(\mathbb N^\dag, B)$ such that $[[\theta]] = x \circ [[\alpha_2]]$ and $[[\rho]] = y \circ [[\alpha_2]]$.  In particular, $[[\ev_m \circ \theta]] = [[\ev_m \circ \rho]]$ for all $m \in \mathbb N^\dag$. Lemma~\ref{lem:factor2} now implies that for all $m \in \mathbb N^\dag$, $\ev_m \circ \theta \circ \alpha_1$ is homotopic to $\ev_m \circ \rho \circ \alpha_1$; let $h_m \colon A_1 \rightarrow C([0, 1], B)$ be such a homotopy.
	
	By, Corollary~\ref{cor:homotopy-stability} there are a finite set $\mathcal G \subseteq A_1$ and $\delta > 0$ such that if $D$ is a $C^*$-algebras and $\phi, \psi \colon A_1 \rightarrow D$ are $^*$-homomorphisms with $\| \phi(a) - \psi(a)\| < \delta$ for all $a \in \mathcal G$, then $\phi \circ \alpha_0$ and $\psi \circ \alpha_0$ are homotopic.  Fix $m_0 \in \mathbb N$ such that
	\[ \| \theta(\alpha_1(a))(m) - h_\infty(a)(0) \| < \delta \qquad \text{and} \qquad \|\rho(\alpha_1(a))(m) - h_\infty(a)(1) \| < \delta \]
	for all $a \in \mathcal G$ and $m \geq m_0$.  Define $D = C(\mathbb N^\dag_{\geq m_0}, B)$ and $\phi, \psi \colon A_1 \rightarrow D$ by $\phi(a)(m) = \theta(\alpha_1(a))(m)$ and $\psi(a)(m) = \theta(\alpha_1(a))(m)$ for all $a \in A_1$ and $m \in \mathbb N^\dag$ with $m \geq m_0$.  By the choice of $\mathcal G$ and $\delta$, $\phi \circ \alpha_0$ and $\psi \circ \alpha_0$ are homotopic; let $k \colon A \rightarrow C([0, 1], D)$ be such a homotopy.  
	
	Finally, define a homotopy $h \colon A_0 \rightarrow C([0, 1], C(\mathbb N^\dag, B))$ by
	\[ h(a)(s)(m) = \begin{cases} h_m(\alpha_0(a))(s) & m < m_0 \\ k(a)(s)(m) & m \geq m_0 \end{cases} \]
	for all $a \in A_0$, $s \in [0, 1]$, and $m \in \mathbb N^\dag$.  Then $h$ defines a homotopy from $\theta \circ \alpha_1 \circ \alpha_0$ to $\rho \circ \alpha_1 \circ \alpha_0$.  Therefore,
	\[ x \circ [[\alpha]] = [[\theta \circ \alpha_1 \circ \alpha_0]] = [[\rho \circ \alpha_1 \circ \alpha_0]] = y \circ [[\alpha]], \]
	as required.
\end{proof}

The following result establishes the main properties of the topology on the space of (Hausdorffized) asymptotic morphisms.

\begin{theorem}\label{thm:hausdorff}
	If $A$ and $B$ are $C^*$-algebras with $A$ separable, then $[[A, B]]_{\mathrm{Hd}}$ is a projective limit of discrete topological spaces.  These  may be taken to be countable if $B$ is separable.  In particular,  $[[A, B]]_{\mathrm{Hd}}$ is totally disconnected and completely metrizable, and it is separable if $B$ is separable.  Moreover, if $(\underline A, \underline \alpha)$ is a shape system for $A$, then the maps $\alpha_{\infty, n}^* \colon [[A, B]] \rightarrow \Homah{A_n}{B}$ of Lemma~\ref{lem:factor} induce a homeomorphism
	\[ \begin{tikzcd}\relax
		[[A, B]]_{\mathrm{Hd}} \arrow{r}{\cong} & \varprojlim\, (\Homah{A_n}{B}, \alpha_n^*). 
	\end{tikzcd} \]
\end{theorem}

\begin{proof}
	Fix a shape system $(\underline A, \underline \alpha)$ for $A$.  For $n \in \mathbb N$,  Proposition~\ref{prop:discrete-factor} gives a discrete space $X_n$ (which we may take to be countable if $B$ is separable) and continuous maps $g_{n+1} \colon \Homah{A_{n+1}}{B} \rightarrow X_n$ and $f_n \colon X_n \rightarrow \Homah{A_n}{B}$ such that for all $^*$-homomorphisms $\phi \colon A_{n+1} \rightarrow B$, $f_n(g_{n+1}([[\phi]])) = [[\phi \circ \alpha_n]]$.  Then there is a commuting diagram
	\[ \begin{tikzcd}[column sep = small]
		 \cdots \arrow{rr}{\alpha_3^*} \arrow{rd}{g_4} & &  \Homah{A_3}{B} \arrow{rr}{\alpha_2^*} \arrow{rd}{g_3} & & \Homah{A_2}{B} \arrow{rr}{\alpha_1^*} \arrow{rd}{g_2} & & \Homah{A_1}{B} \\
		  \cdots \ar[r] & X_3 \arrow{rr}[swap]{g_3 \circ f_3} \arrow{ru}[swap]{f_3} & &  X_2 \arrow{rr}[swap]{g_2 \circ f_2} \arrow{ru}[swap]{f_2} & &  X_1 \arrow{ru}[swap]{f_1} &
	\end{tikzcd} \]
	in the category of topological spaces.  It follows that
	\[ \varprojlim\, (\Homah{A_n}{B}, \alpha_n^*) \cong \varprojlim\, (X_n, g_n \circ f_n).\]
	Therefore, we only need to prove the final part of the theorem.
	
	The maps $\alpha_{\infty, n}^* \colon [[A, B]] \rightarrow \Homah{A_n}{B}$ induce a continuous map
	\[ \alpha^* \colon [[A, B]] \rightarrow \varprojlim\, (\Homah{A_n}{B}, \alpha_n^*), \]
	and, since the inverse limit is Hausdorff (and, in particular, $T_0$), $\alpha^*$ factors through a continuous map
	\[ \bar\alpha^* \colon [[A, B]]_{\mathrm{Hd}} \rightarrow \varprojlim\, (\Homah{A_n}{B}, \alpha_n^*). \] 
	We will show that $\bar\alpha^*$ is a homeomorphism.
	
	First we prove that $\bar\alpha^*$ is injective.  Suppose that $\phi, \psi \colon A \xrightarrow\approx B$ are asymptotic morphisms with $\alpha^*([[\phi]]) = \alpha^*([[\psi]])$.  By Theorem~\ref{thm:diagram-rep}, there are strong homotopy morphisms $(\underline \phi, \underline h), (\underline \psi, \underline k) \colon (\underline A, \underline \alpha) \rightarrow B$ with homotopy limits $\phi$ and $\psi$, respectively.  For each $n \in \mathbb N$,
	\[ [[\phi_{n+1}]] = [[\phi \circ \alpha_{\infty, n+1}]] = [[\psi \circ \alpha_{\infty, n+1}]] = [[\psi_{n+1}]], \]
	where the middle inequality follows from $\alpha^*([[\phi]]) = \alpha^*([[\psi]])$ and the outer two equalities follow from Proposition~\ref{prop:limit-factor}.  By the definition of a strong homotopy morphism of inductive systems, $\phi_{n+1} \circ \alpha_n$ and $\psi_{n+1} \circ \alpha_n$ are homotopic to $\phi_n$ and $\psi_n$, respectively.  Also, by Lemma~\ref{lem:factor2}, $\phi_{n+1} \circ \alpha_n$ and $\psi_{n+1} \circ \alpha_n$ are homotopic.  Concatenating these homotopies shows that $\phi_n$ and $\psi_n$ are homotopic.  Then Proposition~\ref{prop:homotopy-limit} implies $[[\phi]]_{\mathrm{Hd}} = [[\psi]]_{\mathrm{Hd}}$, so $\bar\alpha^*$ is injective.  
		
	Second, we prove surjectivity. Let $(\phi_n \colon A_n \rightarrow B)_{n=1}^\infty$ be a sequence of $^*$-ho\-mo\-mor\-phisms such that $[[\phi_n]] = [[\phi_{n+1} \circ \alpha_n]]$ for all $n \in \mathbb N$.  For $n \geq N$, Lemma~\ref{lem:factor2} implies $\phi_{n+1} \circ \alpha_n$ is homotopic to $\phi_{n+2} \circ \alpha_{n+2, n}$.  If $h_n \colon A_n \rightarrow C([0, 1],B)$ denotes such a homotopy, then the sequences $(\phi_{n+1} \circ \alpha_n)_{n=1}^\infty$ and $(h_n)_{n=1}^\infty$ form a strong homotopy morphism $(\underline A, \underline \alpha) \rightarrow B$.  If $\phi \colon A \xrightarrow\approx B$ denotes the homotopy limit of these sequences, then Proposition~\ref{prop:limit-factor} implies $[[\phi \circ \alpha_{\infty, n}]] = [[\phi_{n+1} \circ \alpha_n]] = [[\phi_n]]$ for all $n \in \mathbb N$.  So $\alpha^*$ is surjective, and hence so is $\bar\alpha^*$.
	
	It remains to prove continuity of the inverse of $\bar\alpha^*$.  To this end, consider $x \in [[A, B]]$ and a sequence $(x_m)_{m=1}^\infty \subseteq [[A, B]]$ such that $\alpha_{\infty, n}^*(x_m) \rightarrow \alpha_{\infty, n}^*(x)$ in $\Homah{A_n}{B}$ for all $n \in \mathbb N$.  By Proposition~\ref{prop:inclusion-continuous}, $\alpha_{\infty, n}^*(x_m) \rightarrow \alpha_{\infty, n}^*(x)$ in $[[A_n, B]]$ for all $n \in \mathbb N$. 
	Then by Theorem~\ref{thm:asymp-convergence}, there is a $y_{n+1}' \in [[A_{n+1}, C(\mathbb N^\dag, B)]]$ such that $y_{n+1}'(m) = \alpha_{\infty, n+1}^*(x_m)$ for $m \in \mathbb N$ and $y_{n+1}'(\infty) = \alpha_{\infty, n+1}^*(x)$.  Define $y_n = y_{n+1}' \circ [[\alpha_n]]$ for $n \in \mathbb N$.  
	Note that $y_{n+1}'(m) = y_{n+2}'(m) \circ [[\alpha_{n+1}]]$ for all $m \in \mathbb N$, so Lemma~\ref{lem:seq-pt-wise} implies $y_n = y_{n+1} \circ [[\alpha_n]]$.  Also, Lemma~\ref{lem:factor} implies $y_n \in \Homah{A_n}{C(\mathbb N^\dag, B)} \subseteq [[A_n, C(\mathbb N^\dag, B)]]$ for all $n \in \mathbb N$.  The surjectivity of $\alpha^*$, applied with $C(\mathbb N^\dag, B)$ in place of $B$, implies there is a $y \in [[A, C(\mathbb N^\dag, B)]]$ such that $\alpha_{\infty, n}^*(y) = y_n$ for all $n \in \mathbb N$.  For $m \in \mathbb N^\dag$ and $n \in \mathbb N$, we have $\alpha_{\infty, n}^*(y(m)) = y_n(m)$.  So the injectivity of $\bar\alpha^*$ implies $y(m) \sim_{\mathrm{Hd}} x_m$ for all $m \in \mathbb N$ and $y(\infty) \sim_{\mathrm{Hd}} x$.  Therefore, $x_n \rightarrow x$ in $[[A, B]]_{\mathrm{Hd}}$.  Hence $\bar\alpha^*$ is a homeomorphism.
\end{proof}

\subsection{Compatibility with direct limits}\label{sec:limits}

The remainder of the section is devoted to proving the continuity of the functor $[[\,\cdot\,,B]]_{\mathrm{Hd}}$ on separable $C^*$-algebras (Theorem~\ref{thm:t0-continuity}).  When restricting to shape systems, this continuity result follows easily from the projective limit decomposition in Theorem~\ref{thm:hausdorff} and the factorization result in Lemma~\ref{lem:factor}.  The general case will reduce to this case by approximating a given inductive system by a shape system with the same limit---the precise statement needed is given in Lemma~\ref{lem:double-lim}.  We start with the following approximation lemma, which is a typical application of semiprojectivity.   

\begin{lemma}\label{lem:shape-map}
	Suppose $(\underline A, \underline \alpha)$ is a shape system with limit $A$ and $(\underline B, \underline \beta)$ is an inductive system with limit $B$ such that each $\beta_n$ is surjective.  For $n \in \mathbb N$, let $\mathcal F_n \subseteq A_n$ be a finite set and let $\epsilon_n > 0$. If $\phi \colon A \rightarrow B$ is a $^*$-homomorphism, then there are a strictly increasing function $f \colon \mathbb N \rightarrow \mathbb N$ and $^*$-homomorphisms $\phi_n \colon A_n \rightarrow B_{f(n)}$ such that the diagram
	\[\begin{tikzcd}[column sep = large]
		A_1 \arrow{r}{\alpha_1} \arrow{d}[swap]{\phi_1} & A_2 \arrow{r}{\alpha_2} \arrow{d}[swap]{\phi_2} & A_3 \arrow{r}{\alpha_3} \arrow{d}[swap]{\phi_3} & \cdots &[-40pt] A \arrow{d}{\phi} \\
		B_{f(1)} \arrow{r}[swap]{\beta_{f(2),f(1)}} & B_{f(2)} \arrow{r}[swap]{\beta_{f(3), f(2)}} & B_{f(3)} \arrow{r}[swap]{\beta_{f(4), f(3)}} & \cdots & B
	\end{tikzcd} \]
	commutes up to homotopy,
	\[ \max_{a \in \mathcal F_n}\, \|\beta_{f(n+1), f(n)}(\phi_n(a)) - \phi_{n+1}(\alpha_n(a)) \| < \epsilon_n, \]
	and $\phi \circ \alpha_{\infty, n} = \beta_{\infty, f(n)} \circ \phi_n$ for all $n \in \mathbb N$.
\end{lemma}

\begin{proof}
	For $n \in \mathbb N$, Corollary~\ref{cor:homotopy-stability} and the semiprojectivity of $\alpha_n$ imply there are a finite set $\mathcal G_{n+1} \subseteq A_{n+1}$ and $\delta_{n+1} > 0$ such that for all $C^*$-algebras $D$ and $^*$-ho\-mo\-mor\-phisms $\theta, \rho \colon A_{n+1} \rightarrow D$, if $\|\theta(a) - \rho(a)\|< \delta_{n+1}$ for all $a \in \mathcal G_{n+1}$, then $\theta \circ \alpha_n$ is homotopic to $\rho\circ \alpha_n$.  By enlarging the sets $\mathcal G_n$ and decreasing $\delta_n$, we may assume $\alpha_n(\mathcal F_n) \subseteq \mathcal G_{n+1}$ and $\delta_{n+1} < \epsilon_n$ for all $n \in \mathbb N$.
	
	It suffices to construct a strictly increasing function $f \colon \mathbb N \rightarrow \mathbb N$ and $^*$-ho\-mo\-morphisms $\psi_{n+1} \colon A_{n+1} \rightarrow B_{f(n)}$ such that 
	\[ \max_{a \in \mathcal G_{n+1}}\, \|\beta_{f(n+1), f(n)}(\psi_{n+1}(a)) - \psi_{n+2}(\alpha_{n+1}(a))\| < \delta_{n+1} \] 
	and $\phi \circ \alpha_{\infty, n+1} = \beta_{\infty, f(n)} \circ \psi_{n+1}$.  Indeed, the $^*$-homomorphisms $\phi_n = \psi_{n+1} \circ \alpha_n$ will satisfy the conditions of the lemma.
	
	Since $\alpha_{\infty, 2} \colon A_2 \rightarrow B$ is semiprojective, there are $f(1) \in \mathbb N$ and a $^*$-homomorphism $\psi_2 \colon A_2 \rightarrow B_{f(1)}$ such that $\beta_{\infty, f(1)} \circ \psi_2 = \phi \circ \alpha_{\infty, 2}$.  Assume $f(1), \ldots, f(n)$ and $\psi_2, \ldots, \psi_{n+1}$ have been constructed.  Because $\alpha_{\infty, n+2}$ is semiprojective, there are $f'(n+1) \in \mathbb N$ and a $^*$-homomorphism $\psi'_{n+2}\colon A_{n+2} \rightarrow B_{f'(n+1)}$ so that 
	\[
		\beta_{\infty, f'(n+1)} \circ \psi'_{n+2} = \phi \circ \alpha_{\infty, n+2}. 
	\]
		  Then $\beta_{\infty, f(n)} \circ \psi_{n+1} = \beta_{\infty, f'(n+1)} \circ \psi_{n+2}' \circ \alpha_{n+1}$, and hence
	\[ \lim_{m \rightarrow \infty} \| \beta_{m, f(n)}( \psi_{n+1}(a)) - \beta_{m, f'(n+1)}(\psi'_{n+2}(\alpha_{n+1}(a))) \| = 0. \]
	Therefore, we may find $f(n+1) \in \mathbb N$ with $f(n+1) > \max\{f(n), f'(n+1) \}$ such that 
	\[ \max_{a \in \mathcal G_{n+1}}\, \| \beta_{f(n+1), f(n)}( \psi_{n+1}(a)) - \beta_{f(n+1), f'(n+1)}(\psi'_{n+2}(\alpha_n(a))) \| < \delta_{n+1}. \]
	Define $\psi_{n+2} = \beta_{f(n+1), f'(n+1)} \circ \psi'_{n+2}$.
\end{proof}

Applying the previous lemma inductively (and taking care with the estimates) yields the following approximation result.  This allows us to replace a given inductive system with a shape system with the same limit.

\begin{lemma}\label{lem:double-lim}
	If $(\underline A, \underline \alpha)$ is an inductive system of separable $C^*$-algebras with limit $A$, then there is a diagram
		\[ \begin{tikzcd}[row sep = 30pt, column sep = 35pt]
		A_1^1 \arrow{r}{\alpha_1^1}\arrow{d}[swap]{\beta_1^1} \arrow{dr}[pos = .4]{\gamma_1} & A_2^1 \arrow{r}{\alpha_2^1}\arrow{d}[swap]{\beta_2^1} & A_3^1 \arrow{r}{\alpha_3^1}\arrow{d}[swap]{\beta_3^1} & \cdots &[-38pt] \\
		A_1^2\arrow{r}{\alpha_1^2}\arrow{d}[swap]{\beta_1^2} & A_2^2 \arrow{r}{\alpha_2^2}\arrow{d}[swap]{\beta_2^2}\arrow{dr}[pos = .4]{\gamma_2} & A_3^2\arrow{r}{\alpha_3^2}\arrow{d}[swap]{\beta_3^2} & \cdots & \\
		A_1^3\arrow{r}{\alpha_1^3}\arrow{d}[swap]{\beta_1^3} & A_2^3 \arrow{r}{\alpha_2^3}\arrow{d}[swap]{\beta_2^3} & A_3^3\arrow{r}{\alpha_3^3}\arrow{d}[swap]{\beta_3^3}\arrow{dr}[pos = .4]{\gamma_3} & \cdots & \\
		\phantom{A_1^4} \arrow[phantom]{d}[pos = -1.5, description]{\rotatebox{90}{.\,.\,.\,}} & \phantom{A_2^4} \arrow[phantom]{d}[pos = -1.5, description]{\rotatebox{90}{.\,.\,.\,}} & \phantom{A_3^4} \arrow[phantom]{d}[pos = -1.5, description]{\rotatebox{90}{.\,.\,.\,}} & \phantom{A_4^4} \arrow[phantom]{dr}[pos = -1.5, description]{\rotatebox{-45}{.\,.\,.\,}} \\[-25pt]
		A_1 \arrow{r}{\alpha_1} & A_2 \arrow{r}{\alpha_2} & A_3 \arrow{r}{\alpha_3} & \cdots & A
	\end{tikzcd} \]
	of $C^*$-algebras and $^*$-homomorphisms that commutes up to homotopy and satisfies the following conditions:
	\begin{enumerate}
		\item\label{double-lim:col} for all $n \in \mathbb N$, the $n$th column $\big((A_n^m, \beta_n^m)\big)_{m=1}^\infty$ is a shape system for $A_n$;
		\item\label{double-lim:diag} the diagonal $\big((A_n^n, \gamma_n)\big)_{n=1}^\infty$ is a shape system for $A$;
		\item\label{double-lim:rectangle} for all $n, m \in \mathbb N$, $\alpha_n \circ \beta_n^{\infty, m} = \beta_{n+1}^{\infty, m} \circ \alpha_n^m $; 
		\item\label{double-lim:triangle} for all $n \in \mathbb N$, $\gamma_{\infty, n} = \alpha_{\infty, n} \circ \beta_n^{\infty, n}$.
	\end{enumerate}
\end{lemma}

\begin{proof}
	The $C^*$-algebras $A_n^m$ and the $^*$-homomorphisms $\alpha_n^m$ and $\beta_n^m$ will be constructed by induction on $n$ and we will define $\gamma_n = \beta_{n+1}^n \circ \alpha_n^n$.  In order to account for the diagonal inductive limit, we will also arrange for each $\beta_n^m$ to be surjective and for each square to approximately commute.	In more detail, we arrange for finite sets $\mathcal F_n^m \subseteq A_n^m$ and finite sets $\mathcal G_{n, j} \subseteq A_n^n$ for $j, m, n \in \mathbb N$ such that
	\begin{enumerate}\setcounter{enumi}{4}
		\item\label{finite-set:F-inc} $\alpha_n^m(\mathcal F_n^m) \subseteq \mathcal F_{n+1}^m$ and $\beta_n^m(\mathcal F_n^m) \subseteq \mathcal F_n^{m+1}$ for all $m, n \in \mathbb N$;
		\item\label{finite-set:G-inc} $\mathcal G_{n,j} \subseteq \mathcal G_{n,j+1}$ and $\gamma_n(\mathcal G_{n, j}) \subseteq \mathcal G_{n + 1, j}$ for all $j, n \in \mathbb N$;
		\item\label{finite-set:G-dense} $\bigcup_{j=1}^\infty \mathcal G_{n, j}$ is dense in $A_n^n$ for all $n \in \mathbb N$;
		\item\label{finite-set:G-to-F} $\mathcal G_{n, n} \subseteq \mathcal F_n^n$;
		\item\label{finite-set:commute} $\displaystyle \max_{a \in \mathcal F_n^m} \, \|\alpha_n^{m+1}(\beta_n^m(a)) - \beta_{n+1}^m(\alpha_n^m(a)) \| < 2^{-(n+m)}$ for all $m, n \in \mathbb N$.
	\end{enumerate}
	
	Let $\big((A_1^m, \beta_1^m)\big)_{m=1}^\infty$ be a shape system for $A_1$ such that each $\beta_1^m$ is surjective (see \cite[Theorem~4.3]{Blackadar85}).  Choose any increasing sequence of finite subsets $(\mathcal G_{1,j})_{j=1}^\infty$ of $A_1^1$ with dense union, and, for $m \in \mathbb N$, define $\mathcal F_1^m = \beta_1^{m, 1}(\mathcal G_{1, 1})$.  Assume $n \in \mathbb N$ and the first $n$ columns have been constructed.  Let $\big((A_{n+1}^m, \beta_{n+1}^m)\big)_{m=1}^\infty$ be a shape system for $A_{n+1}$.  After passing to a subsequence of $\big((A_{n+1}^m, \beta_{n+1}^m)\big)_{m=1}^\infty$, Lemma~\ref{lem:shape-map} provides the maps $\alpha_n^m$ satisfying the required homotopy-commuting property, \ref{double-lim:rectangle}, and \ref{finite-set:commute}.  Define $\gamma_n = \beta_{n+1}^m \circ \alpha_n^n$.  Choose finite sets $\mathcal G_{n+1, j} \subseteq A_{n+1}^{n+1}$ satisfying \ref{finite-set:G-inc} and \ref{finite-set:G-dense}.  Finally, choose finite sets $\mathcal F_{n+1}^m \subseteq A_{n+1}^m$ satisfying \ref{finite-set:F-inc} and \ref{finite-set:G-to-F}.
	
	We have now constructed the homotopy commuting diagram in the statement of the lemma satisfying conditions \ref{double-lim:col} and \ref{double-lim:rectangle}, and also with the properties that each $\beta_n^m$ is surjective, $\gamma_n = \beta_{n+1}^n \circ \alpha_n^n$, and conditions \ref{finite-set:F-inc}--\ref{finite-set:commute} hold.  For $n \in \mathbb N$, note that $\gamma_n$ is semiprojective since $\beta_{n+1}^n$ is, and hence $\big((A_n^n, \gamma_n)\big)_{n=1}^\infty$ forms a shape system.  The definition of $\gamma_n$ and \ref{double-lim:rectangle} implies there is a commuting diagram
	\[ \begin{tikzcd}
		A_1^1 \arrow{r}{\gamma_1} \arrow{d}[swap]{\beta_1^{\infty, 1}} & A_2^2 \arrow{r}{\gamma_2} \arrow{d}[swap]{\beta_2^{\infty, 2}} & A_3^3 \arrow{r}{\gamma_3} \arrow{d}[swap]{\beta_3^{\infty, 3}} & \cdots \\
		A_1 \arrow{r}[swap]{\alpha_1} & A_2 \arrow{r}[swap]{\alpha_2} & A_3 \arrow{r}[swap]{\alpha_3} & \cdots
	\end{tikzcd} \]
	that induces a $^*$-homomorphism $\beta \colon \varinjlim\,(A_n^n, \gamma_n) \rightarrow A$ with $\beta \circ \gamma_{\infty, n} = \alpha_{\infty, n} \circ \beta_n^{\infty, n}$.  It is enough to prove $\beta$ is an isomorphism: indeed, if so, then after using $\beta$ to identify $\varinjlim\,(A_n^n, \gamma_n)$ with $A$, both \ref{double-lim:diag} and \ref{double-lim:triangle} hold.
	
	Because $\beta_n^m$ is surjective for all $m, n \in \mathbb N$, we have $\beta_n^{\infty, n}$ is surjective for all $n \in \mathbb N$.  Therefore, $\beta$ is surjective.  To prove injectivity, it suffices to show that for all $n \in \mathbb N$, $a \in A_n^n$ with $\beta(\gamma_{\infty, n}(a)) = 0$, and $\epsilon > 0$, there is an $m \in \mathbb N$ such that $\|\gamma_{m, n}(a)\| < 4\epsilon$.  For any such $n$, $a$, and $\epsilon$, we have $\alpha_{\infty, n}(\beta_n^{\infty, n}(a)) = 0$, and hence there is $j \in \mathbb N$ with $j > n$ such that $\|\alpha_{j, n}(\beta_n^{\infty, n}(a)) \| < \epsilon$.  After increasing $j$, we may assume $4^{1-j} < \epsilon$.  Further, by \ref{finite-set:G-inc} and \ref{finite-set:G-dense} (increasing $j$ if necessary), we may assume there is a $b \in \mathcal G_{n, j}$ with $\|a - b\| < \epsilon$.  Conditions \ref{finite-set:G-inc} and \ref{finite-set:G-to-F} yield $\gamma_{j, n}(b) \in \mathcal F_j^j$.  Then the definition of $\gamma_n$ together with \ref{finite-set:F-inc} and \ref{finite-set:commute} imply that for each $m \in \mathbb N$ with $m > j$,
	\begin{equation}\label{eq:finite-triangle}
		\|\alpha_{m, j}^m(\beta_j^{m,j}(\gamma_{j, n}(b))) - \gamma_{m, n}(b)\| < 4^{1- j}  < \epsilon.
	\end{equation}
	Since $\|\alpha_{j, n}(\beta_n^{\infty, n}(a)\| < \epsilon$ and $\|a - b\| < \epsilon$, the definition of $\gamma_{n, j}$ and \ref{double-lim:rectangle} imply $\|\beta_j^{\infty, j}(\gamma_{j, n}(b))\| < 2 \epsilon$.  Hence there is $m \in \mathbb N$ with $\|\beta_j^{m, j}(\gamma_{j, n}(b)\| < 3 \epsilon$.  Then \eqref{eq:finite-triangle} yields $\|\gamma_{m, n}(b)\| < 3 \epsilon$, and hence $\|\gamma_{m, n}(a)\| < 4 \epsilon$, as required.  This shows $\beta$ is injective and completes the proof.
\end{proof}

With the lemmas above, we now prove continuity of $[[\,\cdot\,,B]]_{\mathrm{Hd}}$. Note that this fails for $[[\,\cdot\,, B]]$ in general.  For example, $[[S(\,\cdot\,)\otimes \mathcal K, SB \otimes \mathcal K]] = E(\,\cdot\,, B)$, and Milnor's $\lim^1$-sequence in $E$-theory provides an obstruction to preserving limits.

\begin{theorem}
  \label{thm:t0-continuity}
	If $(\underline A, \underline \alpha)$ is an inductive system of separable $C^*$-algebras with limit $A$ and $B$ is a $C^*$-algebra, then the $^*$-homomorphisms $\alpha_{\infty, n} \colon A_n \rightarrow A$ induce a homeomorphism 
	\[
	\begin{tikzcd}\relax 
		[[A, B]]_{\mathrm{Hd}} \arrow{r}{\cong} & \varprojlim \, ([[A_n, B]]_{\mathrm{Hd}}, \alpha_n^*). 
	\end{tikzcd} \]
\end{theorem}

\begin{proof}
	Let $X = \varprojlim\, ([[A_n, B]]_{\mathrm{Hd}}, \alpha_n^*)$ and let $f_{n, \infty} \colon X \rightarrow [[A_n, B]]_{\mathrm{Hd}}$ be the canonical maps.	 The maps $\alpha_{\infty, n}^*$ induce a continuous map $f \colon [[A, B]]_{\mathrm{Hd}} \rightarrow X$ with $f_{n, \infty} \circ f = \alpha_{\infty, n}^*$.  We will construct a continuous inverse of $f$.  We adopt the notation from Lemma~\ref{lem:double-lim}.  The diagram in Lemma~\ref{lem:double-lim} induces a commuting diagram
	\[ \begin{tikzcd}[row sep = 40pt]\relax
		\Homah{A_1^1}{B} & \Homah{A_2^1}{B} \arrow{l}[swap]{(\alpha_1^1)^*}  & \Homah{A_3^1}{B} \arrow{l}[swap]{(\alpha_2^1)^*} & \cdots \arrow{l}[swap]{(\alpha_3^1)^*} &[-50pt] \\
		\Homah{A_1^2}{B} \arrow{u}[pos = .6]{(\beta_1^1)^*} & \Homah{A_2^2}{B} \arrow{l}[swap]{(\alpha_1^2)^*} \arrow{u}[pos = .6]{(\beta_2^1)^*} \arrow{ul}[swap, pos = .6]{\gamma_1^*} & \Homah{A_3^2}{B} \arrow{l}[swap]{(\alpha_2^2)^*} \arrow{u}[pos = .6]{(\beta_3^1)^*}  &  \cdots \arrow{l}[swap]{(\alpha_3^2)^*} & \\
		\Homah{A_1^3}{B} \arrow{u}[pos = .6]{(\beta_1^2)^*} & \Homah{A_2^3}{B} \arrow{l}[swap]{(\alpha_1^3)^*} \arrow{u}[pos = .6]{(\beta_2^2)^*} & \Homah{A_3^3}{B} \arrow{l}[swap]{(\alpha_2^3)^*} \arrow{u}[pos = .6]{(\beta_3^2)^*} \arrow{ul}[swap, pos = .6]{\gamma_2^*} & \cdots \arrow{l}[swap]{(\alpha_3^3)^*} & \\
		\phantom{A_1} \arrow{u}[pos = .6]{(\beta_1^3)^*} & \phantom{A_2} \arrow{u}[pos = .6]{(\beta_2^3)^*} & \phantom{A_3} \arrow{u}[pos = .6]{(\beta_3^3)^*} & \phantom{A_4} \arrow{ul}[swap, pos = .6]{\gamma_3^*} &  \\[-33pt]
		[[A_1, B]]_{\mathrm{Hd}} \arrow[phantom]{u}[description, pos = 2]{\rotatebox{90}{.\,.\,.\,}} &\relax [[A_2, B]]_{\mathrm{Hd}} \arrow{l}[swap]{\alpha_1^*} \arrow[phantom]{u}[description, pos = 2]{\rotatebox{90}{.\,.\,.\,}} &\relax [[A_3, B]]_{\mathrm{Hd}} \arrow{l}[swap]{\alpha_2^*} \arrow[phantom]{u}[description, pos = 2]{\rotatebox{90}{.\,.\,.\,}}& \phantom{\Homah{A_4}{B}} \arrow{l}[swap]{\alpha_3^*} &\relax [[A, B]]_{\mathrm{Hd}} \arrow[phantom]{l}[description, pos = 2.5]{\cdots} \arrow[phantom]{ul}[description, pos =2]{\rotatebox{-40}{.\,.\,.}}
	\end{tikzcd} \]
	of topological spaces and continuous maps.  By Theorem~\ref{thm:hausdorff}, the maps $\gamma_{\infty, n}^*$ induce a homeomorphism
	\[ \begin{tikzcd}\relax 
		[[A, B]]_{\mathrm{Hd}} \arrow{r}{\cong} & \varprojlim\, (\Homah{A_n^n}{B}, \gamma_n^*). 
	\end{tikzcd} \]
	Therefore, the maps $(\beta_n^{\infty, n})^* \colon [[A_n, B]]_{\mathrm{Hd}} \rightarrow \Homah{A_n^n}{B}$ induce a continuous map $g \colon X \rightarrow [[A, B]]_{\mathrm{Hd}}$
	such that $\gamma_{\infty, n}^* \circ g = (\beta_n^{\infty, n})^* \circ f_{\infty, n}$ for all $n \in \mathbb N$.
	
	For $n \in \mathbb N$, we have
	\[ \gamma_{\infty, n}^* \circ g \circ f = (\beta_n^{\infty, n})^* \circ f_{\infty, n} \circ f = (\beta_n^{\infty, n})^* \circ \alpha_{\infty, n}^* = \gamma_{\infty, n}^*, \]
	and hence $g \circ f = \mathrm{id}_{[[A, B]]_{\mathrm{Hd}}}$.  In the other direction, to show $f \circ g = \mathrm{id}_X$, it suffices to show $f_{n, \infty} \circ f \circ g = f_{n, \infty}$ for all $n \in \mathbb N$.  To this end, fix $n \in \mathbb N$.  By Theorem~\ref{thm:hausdorff}, the maps $\beta_n^{\infty, m}$, $m \in \mathbb N$, induce a homeomorphism
	\[ \begin{tikzcd}\relax 
		[[A_n, B]]_{\mathrm{Hd}} \arrow{r}{\cong} & \varprojlim\, (\Homah{A_n^m}{B}, (\beta_n^m)^*),
	\end{tikzcd} \]
	so it suffices to show $(\beta_n^{\infty, m})^* \circ f_{n, \infty} \circ f \circ g = (\beta_n^{\infty, m})^* \circ f_{n, \infty}$ for all $m \in \mathbb N$ with $m > n$.  For such $m$, we compute
	\begin{align*}
		   (\beta_n^{\infty, m})^* \circ f_{n, \infty} \circ f \circ g 
		&= (\beta_n^{\infty, m})^* \circ \alpha_{\infty, n}^* \circ g \\
		&= (\alpha^m_{m,n})^* \circ (\beta_m^{\infty, m})^* \circ \alpha_{\infty, m}^* \circ g \\
		&= (\alpha_{m,n}^m)^* \circ \gamma_{\infty, m}^* \circ g \\
		&= (\alpha_{m, n}^m)^* \circ (\beta_m^{\infty, m})^* \circ f_{\infty, m} \\
		&= (\beta_n^{\infty, m})^* \circ \alpha_{m, n}^* \circ f_{\infty, m} \\
		&= (\beta_n^{\infty, m})^* \circ f_{\infty, n},
	\end{align*}
	as required.  So $f \circ g = \mathrm{id}_X$, and $f$ is a homeomorphism.
\end{proof}

\section{Applications}
\label{sec:applications}

Recall from Definition~\ref{def:homotopy-limit} that every strong homotopy morphism  $(f, \underline \phi, \underline h)$ between inductive systems of $C^*$-algebras $(\underline A, \underline \alpha)$ and $(\underline B, \beta)$ induces an asymptotic morphism
\[ \hlim\,(f, \underline \phi, \underline h)  \colon \varinjlim\,(\underline A, \underline \alpha) \xrightarrow\approx \varinjlim\,(\underline B, \underline \beta). \]
Dadarlat showed in \cite{Dadarlat94} that this construction produces a functor $\hoind \rightarrow \AM$ on a suitable category $\hoind$ (the objects of which are inductive systems of separable $C^*$-algebras), and restricts to an equivalence on the \emph{strong shape category}, defined as the full subcategory $\ssh$ of $\hoind$ whose objects of which are shape systems (Theorem~\ref{thm:hlim-equiv}).  We will combine this result with our work in the previous sections to show that the Hausdorffized asymptotic category $\AMH$ (see Definition~\ref{def:hd-asym-cat}) is equivalent to the shape category $\sh$ (Theorem~\ref{thm:haus-functor-equiv}).  This will prove Theorem~C.

In Section~\ref{sec:e-theory}, we apply these results to obtain a topology on $E$-theory.  In particular, we prove Theorems~\ref{thm:main} and~\ref{thm:cont-EL} from the introduction.  Some other properties of $E$-theory mentioned in the introduction are also discussed.
	
\subsection{The Hausdorffized homotopy limit functor}
\label{sec:hausd-homot-limit}

We first establish some notation.  There are several categories appearing in this subsection.  For convenience, abbreviated forms of the definitions are collected in Table~\ref{tab:categories}.  

Let $\cstar$ be the category of separable $C^*$-algebras and $^*$-homomorphisms.  If $A$ and $B$ are $C^*$-algebras, we write $[A, B]$ for the set of homotopy equivalence classes of $^*$-homomorphisms $A\rightarrow B$.  Let $\ho$ be the category whose objects are separable $C^*$-algebras and where the morphisms $A\to B$ are the elements of $[A, B]$.

For any category $\mathsf D$,  we let $\ind{D}$ denote the category of
(sequential) inductive systems in $\mathsf D$ modulo passing to subsequences.  More precisely, objects in $\ind{D}$ are given by inductive systems $(\underline A, \underline \alpha)$ in $\mathsf D$ and a morphism $(f, \underline\phi) \colon (\underline A, \underline \alpha) \rightarrow (\underline B, \underline \beta)$ is represented by a pair consisting of a strictly increasing function $f \colon \mathbb N \rightarrow \mathbb N$ and a sequence $\underline \phi = (\phi_n \colon A_n \rightarrow B_{f(n)})_{n=1}^\infty$ of morphisms such that for all $n \in \mathbb N$, $\beta_{f(n+1), f(n)} \circ \phi_n = \phi_{n+1} \circ \alpha_n$.  Two morphisms 
\[ (f, \underline \phi), (g, \underline \psi) \colon (\underline A, \underline \alpha) \rightarrow (\underline B, \underline \beta) \]
are \emph{equivalent}, written as $(f, \underline \phi) \cong (g, \underline \beta)$, if for all $n \in\mathbb N$, there is $m \in \mathbb N$ such that $m > \max\{f(n), g(n)\}$ and $\beta_{m, f(n)} \circ \phi_n = \beta_{m, g(n)} \circ \psi_n$.  Note that if $\mathsf D$ is closed under sequential inductive limits, then there is a canonical \emph{inductive limit functor} $\varinjlim \colon \ind{D} \rightarrow \mathsf D$.  In particular, this is the case for $\mathsf D = \cstar$.

In the special case $\mathsf D = \ho$, we will always view morphisms in $\indho$ as being represented by a homotopy morphism $(f, \underline \phi)$ of inductive systems (see Definition~\ref{def:homotopy-morphisms}) and write the equivalence class of such a morphism as $[f, \underline \phi]$.  Therefore, given homotopy morphisms
\[
  (f, \underline \phi), (g, \underline \psi) \colon (\underline A, \underline \alpha) \rightarrow (\underline B, \underline \beta),
\]
we have $[f, \underline \phi] = [g, \underline \beta]$ if and only if for all $n \in \mathbb N$, there is $m \in \mathbb N$ such that $m > \max\{f(n), g(n)\}$ and the $^*$-homomorphisms $\beta_{\infty, f(n)} \circ \phi_n$  and $\beta_{\infty, g(n)} \circ \psi_n$ are homotopic.

As pointed out in \cite[Section~1.4]{Dadarlat94}, the inductive limit functor does not descend to a functor $\indho \to \ho$.  Dadarlat addressed this by using a variation of the category $\indho$, called $\hoind$, which may be regarded as the homotopy category of $\indC$.  Briefly, objects are given by (sequential) inductive systems of separable $C^*$-algebras (as in $\indC$), and morphisms are equivalence classes $[[f, \underline \phi, \underline h]]$ of strong homotopy morphisms of inductive systems.  We refer the reader to \cite[Definition 3.6]{Dadarlat94} for the precise definition of the equivalence relation, which is a modified version of the equivalence in $\indho$ that accounts for the extra data given by the sequence of homotopies $\underline h$.  In particular, any morphism  in $\indC$ induces a morphism in $\hoind$ with constant homotopies, yielding a functor $\hoinclude\colon \indC\to \hoind$ (which is the identity on objects.)

For our purposes, the most relevant property of $\hoind$ is that the homotopy limit construction $\hlim$ induces a functor $\hoind \rightarrow \AM$, which we continue to write as $\hlim$. It is proved in \cite[Section~2]{Dadarlat94} that
\[ \begin{tikzcd}
    \indC \ar[r, "\hoinclude"] \ar[d, swap, "\varinjlim"] & \hoind \arrow{d}{\hlim} \\
    \cstar \arrow{r}{\rm As} & \AM
  \end{tikzcd}
\] 
commutes, where $\asinclusion$ is the identity on objects and $\asinclusion(\phi) = [[\phi]]$ for every $^*$-ho\-mo\-morphism $\phi$.

The \emph{shape category} $\sh$ and \emph{strong shape category}
$\ssh$ are defined as the full subcategories $\sh \subseteq \indho$ and $\ssh \subseteq \hoind$ whose objects are given by shape systems.  The following result is due to Dadarlat.

\begin{table}
  \centering\renewcommand{\arraystretch}{1.3}
  \begin{tabular}{l p{.3\linewidth} p{.4\linewidth}}
    \toprule
    Category & Objects & Morphisms\\\midrule
    $\cstar$ & Separable $C^*$-algebras & $\phi$, $^*$-homomorphisms\\
    $\AM$ & Separable $C^*$-algebras & $[[\phi]]\in [[A,B]]$, asymptotic homotopy classes of $^*$-homomorphisms\\
    $\AMH$ & Separable $C^*$-algebras & $[[\phi]]_{\mathrm{Hd}}\in [[A,B]]_{\mathrm{Hd}}$, Haus\-dorff\-ized asymptotic homotopy classes of homomorphisms\\
    $\ind{D}$ & Inductive systems in $\mathsf{D}$ & $(f, \underline{\phi})$, pairs of increasing functions on $\mathbb{N}$ and sequences of compatible morphisms in $\mathsf{D}$\\
    $\ho$ & Separable $C^*$-algebras & $[\phi]\in [A,B]$, homotopy classes of $^*$-homomorphisms\\
    $\hoind$ & Inductive systems in $\cstar$ & $[[f, \underline{\phi}, \underline{h}]]$, equivalence classes of strong homotopy morphisms\\
    $\sh$ & Shape systems & Same as $\indho$\\
    $\ssh$ & Shape systems & Same as $\hoind$\\
    \bottomrule\mbox{}\renewcommand{\arraystretch}{1}
  \end{tabular}
  \caption{Some of the categories considered in this section.}
  \label{tab:categories}
\end{table}

\begin{theorem}[{\cite[Theorem~3.7]{Dadarlat94}}]
  \label{thm:hlim-equiv}
	The homotopy limit functor restricts to an equivalence of categories
	\[ \hlim \colon \ssh \xrightarrow\sim \AM. \]
\end{theorem}

Theorem~\ref{thm:haus-functor-equiv} below is the analog of
Theorem~\ref{thm:hlim-equiv} for the categories $\sh$ and $\AMH$.  We
will need some more notation.  Let $\hoforget \colon \hoind
\rightarrow \indho$ be the functor that is the identity on objects and
is defined on morphisms by
\[
  \hoforget([[f, \underline \phi, \underline h]]) = [f, \underline \phi]
\] 
and let  $\hausdorffize \colon \AM \rightarrow \AMH$ be the functor that is the identity on objects and is defined by morphisms by 
\[ \hausdorffize([[\phi]]) = [[\phi]]_{\mathrm{Hd}}. \]

\begin{theorem}
  \label{thm:hlim-diag-commute}
	There is a unique functor $\hlim_{\mathrm{Hd}} \colon \indho
        \rightarrow \AMH$ such that
	\[ \begin{tikzcd}
		\indC \arrow{r}{\hoinclude} \ar[d, swap, "\varinjlim"] & \hoind \arrow{r}{\hoforget} \arrow{d}{\hlim} & \indho \arrow{d}{\hlim_{\mathrm{Hd}}} \\ \cstar \arrow{r}{\rm As} & \AM \arrow{r}{\hausdorffize} & \AMH
	\end{tikzcd} \] 
	commutes.  Explicitly, on objects, $\hlim_{\mathrm{Hd}}(\underline A,
        \underline \alpha) = \varinjlim \, (\underline A, \underline
        \alpha)$, and on morphisms, $\hlim_{\mathrm{Hd}}([f,
        \underline \phi]) = [[\hlim (f, \underline \phi,
        \underline h)]]_{\mathrm{Hd}}$, where
        $(f, \underline \phi, \underline h)$ is a strong homotopy morphism.
\end{theorem}

\begin{proof}
  It suffices to show that $\hlim_{\mathrm{Hd}}$ is well-defined on morphisms
  as the rest follows easily.  This is immediate from the
  well-definedness of $\hlim$ and
  Proposition~\ref{prop:homotopy-limit}.
\end{proof}

The following is the precise version of Theorem~\ref{thm:haus-shape}
from the introduction.

\begin{theorem}
  \label{thm:haus-functor-equiv}
	The Hausdorffized homotopy limit functor induces an equivalence of categories
	\[ \hlim\nolimits_{\mathrm{Hd}} \colon \sh \xrightarrow\sim \AMH. \]
\end{theorem}

\begin{proof}
We need to prove that $\hlim\nolimits_{\mathrm{Hd}}$ is full, dense, and faithful.
  
To prove that $\hlim\nolimits_{\mathrm{Hd}}$ is full means to prove that, given objects $(\underline A, \underline \alpha)$ and $(\underline B, \underline \beta)$ in $\sh$, the map
\begin{equation}
  \label{eq:hlim-full}
  \Hom_{\sh} \big( (\underline A, \underline \alpha), (\underline B, \underline \beta) \big)
  \to
  \Hom_{\AMH}\big( \hlim\nolimits_{\mathrm{Hd}} (\underline A, \underline \alpha), \hlim\nolimits_{\mathrm{Hd}} (\underline B, \underline \beta) \big)
\end{equation}
induced by $\hlim_{\mathrm{Hd}}$ is surjective.  This is straightforward: the functor $\hausdorffize$ is full and the restriction of $\hlim$ to $\ssh$ is an equivalence of categories (Theorem~\ref{thm:hlim-equiv}), so the commutativity of the diagram in Theorem~\ref{thm:hlim-diag-commute} implies that the map in \eqref{eq:hlim-full} is surjective.

To prove that $\hlim\nolimits_{\mathrm{Hd}}$ is dense means to prove that given an object $A$ in $\AMH$, there is an object $(\underline A, \underline \alpha)$ in $\sh$ such that $A$ is isomorphic to $\hlim\nolimits_{\mathrm{Hd}} (\underline A, \underline \alpha)$.  This is immediate from the statement that every separable $C^*$-algebras has a shape system (see \cite[Corollary~4.3]{Blackadar85}).

The fact that $\hlim\nolimits_{\mathrm{Hd}}$ is faithful is more involved.  We
need to prove that, given objects $(\underline A, \underline \alpha)$
and $(\underline B, \underline \beta)$ in $\sh$, the map in
\eqref{eq:hlim-full} is injective.  Write $A = \varinjlim
(\underline A, \underline \alpha)$ and $B = \varinjlim\, (\underline B,
\underline \beta)$.  We may further assume the connecting maps in the shape
system $(\underline B, \underline \beta)$ are surjective.  Indeed, $B$ admits a shape system $(\underline B', \underline \beta')$ with each $\beta'_n$ surjective, and then Theorem~\ref{thm:hlim-equiv} implies $(\underline B, \beta)$ and $(\underline B', \beta')$ are isomorphic in the (strong) shape category.  So we may replace $(\underline B, \underline \beta)$ with $(\underline B', \underline \beta')$.  

Suppose
$(f, \underline \phi), (g, \underline \psi)\colon (\underline A,
\underline \alpha) \to (\underline B, \underline \beta)$ are homotopy
morphisms and that  \[ \hlim\nolimits_{\mathrm{Hd}} ([f, \underline \phi]) =
\hlim\nolimits_{\mathrm{Hd}} ([g, \underline \psi]). \]    
To simplify the
notation, and without loss of generality, we assume $f = g =
\id_{\mathbb{N}}$.  Fix $n\in \mathbb{N}$.  Our goal is to show that that, for large
enough $m\in \mathbb{N}$, $\beta_{m,n}\circ \phi_n$ is homotopic to
$\beta_{m,n} \circ \psi_n$, since this implies $[\underline{\phi}] =
[\underline{\psi}]$ in $[(\underline{A}, \underline{\alpha}),
(\underline{B}, \underline{\beta})]$.  By
Corollary~\ref{cor:homotopy-stability} and the semiprojectivity of
$\alpha_n$, there exist a finite set $\mathcal{G} \subset A_{n+1}$
and $\delta > 0$ such that if $D$ is a $C^*$-algebra and $\theta, \rho \colon A \rightarrow D$ are $^*$-homomorphisms with $\|\theta(a) - \rho(a)\| < \delta$, it must be that
$\theta \circ \alpha_n$ is homotopic to $\rho\circ \alpha_n$.

The hypothesis and Proposition~\ref{prop:limit-factor} imply that
\[ [[\beta_{\infty, n+3}\circ \phi_{n+3}]]_{\mathrm{Hd}} = [[\phi\circ
\alpha_{\infty, n+3}]]_{\mathrm{Hd}} = [[\psi\circ \alpha_{\infty,
  n+3}]]_{\mathrm{Hd}} = [[\beta_{\infty, n+3}\circ \psi_{n+3}]]_{\mathrm{Hd}}. \]
Because $\alpha_{n+2}$ is semiprojective,
Lemma~\ref{lem:t0-equiv-homotopic} implies that $\beta_{\infty,
  n+3}\circ \phi_{n+3} \circ \alpha_{n+2}$ is homotopic to
$\beta_{\infty, n+3}\circ \psi_{n+3} \circ \alpha_{n+2}$.  Moreover,
the fact that $\underline{\phi}$ and $\underline{\psi}$ are homotopy
morphisms implies that $\beta_{\infty, n+2} \circ \phi_{n+2}$ is
homotopic to $\beta_{\infty, n+3}\circ \phi_{n+3} \circ \alpha_{n+2}$,
and that $\beta_{\infty, n+2} \circ \psi_{n+2}$ is
homotopic to $\beta_{\infty, n+3}\circ \psi_{n+3} \circ
\alpha_{n+2}$.

We have that $\beta_{\infty, n+2} \circ \phi_{n+2}$ is homotopic to
$\beta_{\infty, n+2} \circ \psi_{n+2}$.  Hence there is a $^*$-homomorphism $\theta\colon A_{n+2}\to
C([0,1], B)$ such that \[ \ev_0\circ \theta =
\beta_{\infty, n+2} \circ \phi_{n+2} \qquad \text{and} \qquad \ev_1\circ \theta =
\beta_{\infty, n+2} \circ \psi_{n+2}. \]  
Let $\overline{\beta}_m = \id_{C[0,1]}\otimes \beta_m$ and regard $C([0,1], B)$ as the limit of the inductive system $(C([0,1], B_m), \underline{\overline{\beta}})$.  If
$m\in \mathbb{N}$ is large enough (and $m > n+2$), the
semiprojectivity of $\alpha_{n+1}$ provides a $^*$-homomorphism
$\tilde{\theta}\colon A_{n+1}\to C([0,1], B_m)$ with
$\overline{\beta}_{\infty, m} \circ \tilde{\theta} = \theta \circ
\alpha_{n+1}$.  Therefore,
\[
  \beta_{\infty, m} \circ \ev_0 \circ \tilde{\theta}
  = \ev_0 \circ \overline{\beta}_{\infty, m} \circ \tilde{\theta}
  = \ev_0 \circ \theta \circ \alpha_{n+1}
  = \beta_{\infty, n+2} \circ \phi_{n+2} \circ \alpha_{n+1}.
\]
Similarly, $\beta_{\infty, m} \circ \ev_1 \circ \tilde{\theta} =
\beta_{\infty, n+2} \circ \psi_{n+2} \circ \alpha_{n+1}$.  Therefore,
if $m$ is large enough, we have that $\| (\ev_0\circ
\tilde{\theta})(a) - (\beta_{m, n+2}\circ \phi_{n+2}\circ
\alpha_{n+1})(a)\| < \delta$ for all $a\in \mathcal{G}$.  Thus
$\ev_0\circ \tilde{\theta}\circ \alpha_n$ and $\beta_{m, n+2}\circ
\phi_{n+2}\circ \alpha_{n+1} \circ \alpha_n$ are homotopic for large
enough $m$, as are $\ev_1\circ \tilde{\theta}\circ \alpha_n$ and
$\beta_{m, n+2}\circ \psi_{n+2}\circ \alpha_{n+1}\circ \alpha_n$, by
the choice of $\mathcal{G}$ and $\delta$.

Now, $\beta_{m, n+2}\circ \phi_{n+2}\circ \alpha_{n+1}\circ \alpha_n$
and $\beta_{m,n}\circ \phi_n$ are homotopic, and hence so are $\beta_{m,
  n+2}\circ \psi_{n+2}\circ \alpha_{n+1}\circ \alpha_n$ and
$\beta_{m,n}\circ \psi_n$ using that $\underline{\phi}$ and
$\underline{\psi}$ are homotopy morphisms.  Since $\ev_0\circ
\tilde{\theta}\circ \alpha_n$ and $\ev_1\circ \tilde{\theta}\circ
\alpha_n$ are homotopic, we (finally) conclude that $\beta_{m,n}\circ
\phi_n$ and $\beta_{m,n}\circ \psi_n$ are homotopic, as desired.
\end{proof}

Two separable $C^*$-algebras are isomorphic in the category $\sh$ if
and only if they are isomorphic in the category $\ssh$ by
\cite[Theorem~3.9]{Dadarlat94}.  In combination with
Theorem~\ref{thm:haus-functor-equiv}, this gives the following.

\begin{corollary}
  \label{cor:same-isom-objects-asym-haus}
	Two separable $C^*$-algebras are isomorphic in the category $\AM$ if and only if they are isomorphic in the category $\AMH$.  In fact, if $A$ and $B$ are separable $C^*$-algebras and $x \in [[A, B]]$ is such that $\hausdorffize(x) \in [[A, B]]_{\mathrm{Hd}}$ is an isomorphism, then $x$ is an isomorphism.  
\end{corollary}

\subsection{The topology on $E$-theory}\label{sec:e-theory}

In this final subsection, we apply the results of the previous sections to $E$-theory and prove Theorems~\ref{thm:main} and~\ref{thm:cont-EL}.

For a $C^*$-algebra $A$ define $SA =  C_0(\mathbb{R}) \otimes A$, and write $\mathcal K$ for the $C^*$-algebra of compact operators on a separable infinite dimensional Hilbert space.  Given separable
$C^*$-algebras $A$ and $B$, Connes and Higson
defined 
\[
	E(A,B) = [[SA\otimes \mathcal{K}, SB\otimes \mathcal{K}]];
\]
see \cite[Section 4]{Connes-Higson90}. Then $E(A, B)$ is an abelian group with the
sum of $[[\phi]]$ and $[[\psi]]$ given by the orthogonal sum: fix an
isomorphism $\kappa\colon M_2(\mathcal{K})\to \mathcal{K}$ (which is unique up
to homotopy) and define
\[
  [[\phi]] + [[\psi]] = [[\id_{SB}\otimes \kappa]] \circ
  \left[\left[
      \begin{pmatrix}
        \phi & 0 \\ 0 & \psi
      \end{pmatrix}
  \right]\right].
\]

The existence of the topology on $E(A, B)$ promised in Theorem~\ref{thm:main} now follows immediately from the existence of our topology on asymptotic morphisms.

\begin{proof}[Proof of Theorem~\ref{thm:main}]
 This is a
  special case of Theorems~\ref{thm:countability}
  and~\ref{thm:asymp-convergence}.
\end{proof}

As expected, the algebraic operations on $E(A, B)$ are continuous.

\begin{theorem}\label{thm:top-group}
	If $A$, $B $, and $D$ are separable $C^*$-algebras, then $E(A, B)$ is a topological group and the product $E(A,B)\times E(B,D)\to E(A,D)$ is jointly continuous
\end{theorem}

\begin{proof}
	The continuity of the product is immediate from Theorem~\ref{thm:joint-continuity}.  To show $E(A, B)$ is a topological group, it suffices to show the continuity of subtraction.  Suppose $(x_n)_{n=1}^\infty$ and $(y_n)_{n=1}^\infty$ are sequences in $E(A, B)$ with $x_n \rightarrow x$ and $y_n \rightarrow y$ in $E(A, B)$.  Let $\hat x, \hat y \in E(A, C(\mathbb N^\dag, B))$ be such that $\hat x(m) = x_m$, $\hat x(\infty) = x$, $\hat y(m) = y_m$, and $\hat y(\infty) = y$ for all $m \in \mathbb N$.  Then $z = \hat x - \hat y \in E(A, C(\mathbb N^\dag, B))$ satisfies $z(m) = x_m - y_m$ for $m \in \mathbb N$ and $z(\infty) = x - y$.  So $x_m - y_m \rightarrow x - y$, as required.
\end{proof}

For separable $C^*$-algebras $A$ and $B$, Dadarlat \cite{Dadarlat05}
defined a topology on $KK(A,B)$ that is second countable and satisfies
\emph{Pimsner's condition:} a sequence $(x_n)$ in $KK(A,B)$ converges
to $x_\infty$ if and only if there exists $y\in KK(A,
C(\mathbb{N}^\dag, B))$ such that $y(n) = x_n$ for all $n\in
\mathbb{N}$ and $y(\infty) = x_\infty$.  Such a topology is obviously
unique.  Therefore, when $A$ is nuclear (or just $K$-nuclear), so that
$KK(A,B) \cong E(A,B)$ via an isomorphism that respects the product
structure (see \cite{Connes-Higson90} or
\cite[Theorem~25.6.3]{Blackadar98}), Theorem~\ref{thm:main} shows that
the Dadarlat's topology and ours coincide.

Also in \cite{Dadarlat05}, Dadarlat defined $KL(A,B)$ to be the quotient of
$KK(A,B)$ by the closure of $\{0\}$.  (This followed an earlier
definition of R{\o}rdam \cite{Rordam95} that required the UCT.)  In a
similar fashion, we define $EL(A,B)$ to be the quotient of $E(A,B)$ by
the closure of $\{0\}$.  This coincides with $[[SA\otimes \mathcal{K},
SB\otimes \mathcal{K}]]_{\mathrm{Hd}}$ (see Definition~\ref{def:haus-asym}) by
the continuity of addition and inverses (Theorem~\ref{thm:top-group}).  When $A$ is nuclear (or just
$K$-nuclear), $KL(A,B) \cong EL(A,B)$.

\begin{theorem}[cf.\ {\cite[Proposition~2.8]{Dadarlat05}}] If $A$,
  $B$, and $D$ are separable $C^*$-algebras, then $EL(A, B)$ is a
  totally disconnected,  separable, and completely metrizable topological group, and the composition product $EL(A, B) \times EL(B, D) \rightarrow EL(A, D)$ is jointly continuous.
\end{theorem}

\begin{proof}
	The group structure on $EL(A, B)$ follows from Theorem~\ref{thm:top-group}, and Theorem~\ref{thm:hausdorff} implies $EL(A, B)$ is a Polish space.  The continuity of composition is a special case of Proposition~\ref{prop:composition}.
\end{proof}

Now we turn to the proof of Theorem~\ref{thm:cont-EL}, which states that---unlike $E$-theory and $KK$-theory---the group $EL(\,\cdot\,, B)$ always preserves inductive limits.

\begin{proof}[Proof of Theorem~\ref{thm:cont-EL}]
  This follows immediately from Theorem~\ref{thm:t0-continuity} and the remarks above to identify $EL$ and $KL$ in the presence of nuclearity.
\end{proof}

Two (separable) $C^*$-algebras $A$ and $B$ are \emph{$KK$-equivalent}
if there is an invertible element in $KK(A,B)$.  Similar terminology
is used for $KL$, $E$, and $EL$.  Dadarlat showed in
\cite[Corollary~5.2]{Dadarlat05}, using the Kirchberg--Phillips theorem, that two nuclear \mbox{$C^*$-algebras} are
$KK$-equivalent if and only if they are $KL$-equivalent.  The following strengthening of this statement eschews the nuclearity assumption.  In the nuclear setting, it provides a proof of Dadarlat's result that does not depend on classification theorems.

\begin{proposition}
  Two separable $C^*$-algebras are $E$-equivalent if and only if they
  are $EL$-equivalent.
\end{proposition}

\begin{proof}
  This is a special case of
  Corollary~\ref{cor:same-isom-objects-asym-haus}.
\end{proof}

We end by pointing out that the usual tools used to compute $KL$-groups can also be used to compute $EL$-groups.  If a separable $C^*$-algebra $A$ is $E$-equivalent to a commutative $C^*$-algebra, then there is a natural short exact sequence 
\[ 0 \rightarrow \mathrm{Ext}_\mathbb Z^1(K_{*+1}(A), K_*(B)) \rightarrow E(A, B) \rightarrow \Hom_\mathbb Z(K_*(A), K_*(B)) \rightarrow 0, \]
known as the \emph{universal coefficient theorem (UCT)} in $E$-theory.  This can be deduced from the universal coefficient theorem in $KK$-theory of \cite{Rosenberg-Schochet87} by identifying $E(A, B)$ with $E(D, B)$, and then with $KK(D, B)$, for some commutative $C^*$-algebra $D$ (for which the UCT holds).  In a similar way, one can borrow the version of the UCT from \cite{Dadarlat-Loring96} to prove the following.

\begin{theorem}
	If $A$ and $B$ are separable $C^*$-algebras, and $A$ is $E$-equivalent to a commutative $C^*$-algebra, then
	\begin{enumerate}
		\item\label{UCT1} The closure of $\{0\}$ in $E(A, B)$ coincides with the image of the subgroup of $\mathrm{Ext}_\mathbb{Z}^1(K_{*+1}(A), K_*(B))$ consisting of pure extensions.  
		\item\label{UCT2} The natural map $EL(A, B) \rightarrow \Hom_\Lambda(\underline K(A), \underline K(B))$ is an isomorphism of topological groups, where the total $K$-theory groups $\underline K(A)$ and $\underline K(B)$ are endowed with the discrete topology and the space of homomorphisms is equipped with the topology of pointwise convergence.
	\end{enumerate}
\end{theorem}

\begin{proof}
	As noted above, we may identify $E(A, B)$ with $KK(D, B)$ for some commutative $C^*$-algebra $D$ that is $E$-equivalent to $A$.  So it suffices to prove the analogous results in $KK$-theory.  These are known: for example, \ref{UCT1} follows form \cite[Theorem~3.3]{Schochet02} and \ref{UCT2} follows from \cite[Theorem~4.1]{Dadarlat05}.
\end{proof}

\begin{remark}
  By the main result of \cite{manuilov-thomsen04a}, for separable $C^*$-algebras $A$ and $B$, there is a natural isomorphism $E(A, B) \cong KK(SA, Q(B\otimes \mathcal{K}))$, where for a $C^*$-algebra $D$, $M(D)$ denotes the multiplier algebra of $D$, and $Q(D)$ denotes the corona $M(D) / D$.
  One might attempt to use this isomorphism and the topology on $KK$ from \cite{Dadarlat05} to obtain an alternate definition of the topology on $E(A,B)$.
  An immediate technical hurdle that one faces is that the topology on $KK$ from \cite{Dadarlat05} requires the second variable to be separable, whereas $Q(B \otimes \mathcal K)$ is non-separable whenever $B$ is non-zero.
  Separability, as opposed to $\sigma$-unitality, is important to obtain the existence of absorbing representations using the main result of \cite{Thomsen01}, which is a critical ingredient in the definition of the topology in \cite{Dadarlat05}.

  Moreover, if this technical hurdle could be overcome (e.g.\ if the result from \cite{Thomsen01} could be extended to $\sigma$-unital codomains or if one could somehow pass to the limit over separable subalgebras of the codomain, as in \cite[Appendix~B]{cgstw23}), then the topology on $E(A,B)$ obtained using the isomorphism above would be defined in terms of $^*$-homomorphisms
  \begin{displaymath}
    SA \to M(Q(B\otimes \mathcal{K})\otimes \mathcal{K}).
  \end{displaymath}
  This would be a rather difficult and impractical definition to work with.
  The topology introduced in this paper is defined directly in terms of asymptotic morphisms and, in our opinion, is more conceptual and easier to work with than this possible alternative.
  Moreover, it is not clear to the authors how to prove the $E$-theoretic version of Pimsner's condition of Theorem~\ref{thm:main} (or the other properties established in this section) for this potential alternate definition from a topology on $KK(SA, Q(B\otimes \mathcal{K}))$.
  The main difficulty would be understanding the relationship between $C(\mathbb N^\dag, Q(B \otimes \mathcal K))$ and $Q(C(\mathbb N^\dag, B) \otimes K))$.
\end{remark}

\providecommand{\bysame}{\leavevmode\hbox to3em{\hrulefill}\thinspace}
\providecommand{\MR}{\relax\ifhmode\unskip\space\fi MR }
\providecommand{\MRhref}[2]{%
  \href{http://www.ams.org/mathscinet-getitem?mr=#1}{#2}
}
\providecommand{\href}[2]{#2}

\articleend

\title{Corrigendum to ``A topology on $E$-theory''}
\author[J. Carri\'on]{Jos\'e R.\ Carri\'on}
\address{Jos\'e R.\ Carri\'on, Department of Mathematics, Texas Christian
	University, Fort Worth, Texas 76129, USA}
\email{j.carrion@tcu.edu}

\author[C. Schafhauser]{Christopher Schafhauser}
\address{Christopher Schafhauser, Department of  Mathematics, University of Nebraska - Lincoln, Lincoln, Nebraska, USA}
\email{cschafhauser2@unl.edu}

\date{\today}

\maketitle 

The second sentence of \cite[Corollary~4.4]{CS} does not follow from the given reference, and we do not know if it is true as stated.  What is true is that if $\bar x \in [[A, B]]_{\mathrm{Hd}} $ is an isomorphism, then there is an isomorphism $x \in [[A, B]]$ such that $\mathrm{Hd}(x) = \bar x$.  Indeed, \cite[Theorem~1.14]{D} implies every isomorphism in the shape category $\mathsf{sh}$ is induced by an isomorphism in the strong shape category $\mathsf{s}$-$\mathsf{sh}$, and then the result follows from using \cite[Theorem~4.3]{CS} and \cite[Theorem~3.7]{D} to identify these categories with the Hausdorffized asymptotic morphism category $\mathsf{AM}_{\mathrm{Hd}}$ and the asymptotic morphism category $\mathsf{AM}$.

This error has no effect on the rest of the results in the paper.

\addtocontents{toc}{\SkipTocEntry}

\articleend

\end{document}